\documentclass[a4paper,11pt]{amsart}

\usepackage{amsfonts,amsmath,amssymb,amsthm}
\usepackage{mathrsfs}
\usepackage[top=1.5in, bottom=.9in, left=0.9in, right=0.9in]{geometry}

\usepackage[
    hypertexnames=false,
    backref=page,
    pdftex,
    pdfpagemode=UseNone,
    breaklinks=true,
    extension=pdf,
    colorlinks=true,
    linkcolor=blue,
    citecolor=blue,
    urlcolor=blue,
]{hyperref}

\newcommand{\NN}{\mathbb{N}}
\newcommand{\RR}{\mathbb{R}}
\newcommand{\CC}{\mathbb{C}}
\newcommand{\HH}{\mathbb{H}}
\newcommand{\OO}{\mathbb{O}}

\newcommand{\ee}{\mathsf{i}}

\newcommand{\calE}{\mathscr{E}}

\newcommand{\RE}{\mathsf{Re}}
\newcommand{\IM}{\mathsf{Im}}
\newcommand{\Syz}{\mathsf{Syz}}

\newtheorem{teo}{Theorem}[section]
\newtheorem{pro}[teo]{Proposition}
\newtheorem{cor}[teo]{Corollary}
\newtheorem{rem}[teo]{Remark}
\newtheorem{lem}[teo]{Lemma}

\newtheorem*{teo*}{Theorem}
\newtheorem*{exa*}{Example}

\title[Extension and tangential CRF conditions in quaternionic analysis]{Extension and tangential CRF conditions in quaternionic analysis}

\author[M.~Maggesi]{Marco Maggesi\textsuperscript{1}}
\address[\textsuperscript{1}]{Dipartimento di Matematica e Informatica 'Ulisse Dini'
Viale Morgagni, 67/a
50134 FIRENZE}
\email{marco.maggesi@unifi.it}

\author[D.~Pertici]{Donato Pertici\textsuperscript{2}}
\address[\textsuperscript{2}]{Dipartimento di Matematica e Informatica 'Ulisse Dini'
Viale Morgagni, 67/a
50134 FIRENZE}
\email{donato.pertici@unifi.it}

\author[G.~Tomassini]{Giuseppe Tomassini\textsuperscript{3}}
\address[\textsuperscript{3}]{Scuola Normale Superiore, Piazza dei Cavalieri, 7 - I-56126 Pisa, Italy}
\email{giuseppe.tomassini@sns.it}
\thanks{The authors received support from GNSAGA-INdAM and Italian MIUR.}
\date{\today}

\begin{document}

\begin{abstract}
We prove some extension theorems for quaternionic holomorphic
functions in the sense of Fueter.
Starting from the existence theorem for the nonhomogeneous
Cauchy-Riemann-Fueter Problem, we prove that
an $\HH$-valued function $f$ on a smooth hypersurface,
satisfying suitable tangential conditions, is locally a
jump of two $\HH$-holomorphic functions.
From this, we obtain, in particular,
the existence of the solution for the Dirichlet Problem
with smooth data.
We extend these results to the continous case.
In the final part, we discuss the octonian case.
\end{abstract}

\maketitle
\tableofcontents

\section*{Introduction}


This paper aims to set forth the methods of complex analysis in
the quaternionic analysis in several variables.
The main objects of such a theory are the $\HH$-\emph{holomorphic} functions,
i.e., those functions $f=f(q_1,\ldots,q_n)$, $q_1,\ldots,q_n\in\HH$,
which are (left) regular in the sense of Fueter with respect to each variable.
For the basic results in the quaternionic analysis in one and several variables,
we refer to the articles by Sudbery~\cite{S} and~\cite{Pe1} respectively.
As for a more geometric aspect of the theory, we refer to the book~\cite{IMV}
and the rich bibliography quoted there. 

Coming to the content of the paper,
we are dealing with the boundary values and extension problems for $\HH$-holomorphic functions.
As it is well known, this is one of the central themes in complex analysis,
which motivated the study of overdetermined systems of linear partial differential equations,
the CR geometry, and the theory of extension of ``holomorphic objects''.

For the sake of simplicity, we restrict ourselves to the
case $n=2$, even if most of the main results proved in the paper
hold in any dimension.

The paper is organized into three sections.

In Section~\ref{Ge}, after fixing the main notations,
we define the differential forms ${\rm d}q_\alpha$, ${\rm D}{q_\alpha}$
that play a fundamental role, and the Cauchy-Riemann-Fueter operator
$\overline{\mathfrak D}$.
As an application of the Cauchy-Fueter formula in one variable
\cite{Fu1,S},
we prove a result of ``Carleman type'' (Proposition~\ref{Car}).
We also recall the Bochner-Martinelli formula proved in~\cite{Pe1},
and we show that the Bochner-Martinelli kernel ${\bf K}^{BM}(q,q_0)$ writes
as a sum ${\bf K}^{BM}_1(q,q_0)+{\bf K}^{BM}_2(q,q_0)\sf j$,
where ${\bf K}^{BM}_1(q,q_0)$ and
${\bf K}^{BM}_2(q,q_0)$ are complex differential forms and the latter
is exact on $\{q\neq q_0\}$, see~\eqref{KBM2}.

The Section ends with a brief overview of the main results on $\HH$-holomor\-phy,
$\HH$-convexity~\cite{Pe3}, and the
$\overline{\mathfrak D}$-problem~\cite{ABLSS,AL,CSSS,BDS}.

Sections~\ref{TEXI} and~\ref{F1} are the bulk of the paper.
In the first part of Sections~\ref{TEXI},
using the differential forms ${\rm d}q_\alpha$, ${\rm D}{q_\alpha}$,
we formulate the CRF condition on a smooth hypersurface $S$ in terms of
the tangential operators ${{\rm D}q_1}|_S\wedge{\rm d}_{(q_1)}f$,
${{\rm D}q_2}|_S\wedge{\rm d}_{(q_2)}f$ (Theorem~\ref{lu56}).
This allows us to give the notion of \emph{admissible function} $f:S\to\HH$,
which is satisfied by the traces or, more generally, the ``jumps'' of
$\HH$-holomorphic functions, as done by the second author in~\cite{Pe3}.
Admissibility is a second-order condition, so, unlike the complex case, the
traces or, more generally, the jumps of $\HH$-holomorphic functions
satisfy first and second-order equations.
This is not surprising since these problems are related to
local solvability of the Cauchy-Riemann-Fueter Problem $\overline{\mathfrak D}u=g$
and this requires a second-order differential condition for $g$.
The main results of Section~\ref{TEXI} are Theorem~\ref{RH},
and Theorem~\ref{appl1} reported below.

Let $\Omega\subset \HH^2$ be a domain.
A \emph{domain splitting} $(S,U^+,U^-)$ of $\Omega$ is given by a smooth (nonempty)
hypersurface $S$ closed in $\Omega$ and two open disjoint nonempty sets $U^+$, $U^-$,
such that $\Omega\setminus S = U^+ \cup U^-$, where
both $U^+$ and $U^-$ have boundary $S$ in $\Omega$.

We say that a continuous (smooth) function $f:S\to\HH$ is a
\emph{continuous (smooth) jump} relative to a domain
splitting $(S,U^+,U^-)$ of $\Omega$, if there exist two $\HH$-holomorphic
functions $F^+$, $F^-$, on $U^+$, $U^-$ respectively, such that
$F^+$, $F^-$ are continuous (smooth) up to $S$ and $f=F^+|_S-F^-|_S$.

\begin{teo*}
  Let $\Omega\subset\HH^2$ be a convex domain and $(S,U^+,U^-)$ a domain splitting of $\Omega$.
  Let $f:S\to \HH$ a smooth admissible function.
  Then, $f$ is a smooth jump.
\end{teo*}

\begin{teo*}
  Let $\Sigma$ be an open half-space and $S\subset \HH^2$ a
  connected closed smooth hypersurface of $\Sigma$.
  Assume that $\Sigma\setminus S$ splits into two connected
  components $D$ and $W$, with $D$ bounded.
  Let $f:S\cap \Sigma\to\HH$ be a smooth admissible function.
  Then, $f$ extends to $D$ by an $\HH$-holomorphic
  function, which is smooth up to $S$.
\end{teo*}

In Section~\ref{F1}, we extend the
previous results when the function $f$ is admissible in a weak sense.

Finally, in the Appendix,  we provide the characteristic conditions
for the local solvability of the Cauchy-Riemann-Fueter
Problem $\overline{\mathfrak{D}}u=g$ in the case of $n=2$ octonian variables.
This, allows us to generalize some of our constructions and results to the
octonian case.

\section{Generalities}\label{Ge}

In this section, we summarize some of the main notions and results
contained in the seminal papers~\cite{Pe1, Pe2, Pe3}.

\subsection{Fueter operators and \texorpdfstring{$\HH$}{H}-holomorphic functions}

We fix some notations. Let $\HH$ be the quaternion algebra over $\RR$.
For a generic $q\in\HH$ we write
$$
q=\sum_{\alpha=0}^3x_\alpha{\sf i}_\alpha,\qquad
\overline q=x_0-\sum_{\alpha=1}^3x_\alpha{\sf i}_\alpha
$$
$x_\alpha\in\RR$, where ${\sf i_0}=1$, ${\sf i_1}={\sf i}$,
${\sf i_2}={\sf j}$, ${\sf i_3}={\sf k}$.

We also define the following $\HH$-valued differential forms
\begin{equation}\label{giu}
{\rm d}q=\sum_{\alpha=0}^3{\sf i}_\alpha{\rm d}x_\alpha,\qquad
\overline{{\rm d}q}=\sum_{\alpha=0}^3\bar{\sf i}_\alpha{\rm d}x_\alpha
\end{equation}
and
\begin{equation}\label{giu0}
{\rm D}q=\sum_{\alpha=0}^3(-1)^{\alpha}{\sf i}_\alpha{\rm d}X_{\widehat{\alpha}},\qquad
\overline{{\rm D}q}=\sum_{\alpha=0}^3(-1)^{\alpha}\bar{\sf i}_\alpha{\rm d}X_{\widehat{\alpha}},
\end{equation}
where
${\rm d}X_{\widehat\alpha} =
{\rm d}x_0\wedge\cdots\wedge\widehat{{\rm d}x_\alpha}\wedge\cdots\wedge{\rm d}x_3$.

Let $F$ be a ${\rm C}^1$ $\HH$-valued function.
Following Fueter, we define the operators
\begin{equation}
  \label{giu123}
    \frac{\partial F}{\partial q}
    = \sum_{\alpha=0}^3\bar{\sf i}_\alpha\frac{\partial F}{\partial x_\alpha}, \qquad
    \frac{\partial F}{\partial \overline q}
    = \sum_{\alpha=0}^3{\sf i}_\alpha\frac{\partial F}{\partial x_\alpha}.
\end{equation}
We have
\begin{equation}\label{laplaciano}
\Delta F =
\frac{\partial}{\partial q}
\frac{\partial}{\partial \bar q} F
= \frac{\partial}{\partial \bar q}
\frac{\partial}{\partial q} F,
\end{equation}
\begin{equation}\label{lu1}
{\rm d}\big({\rm D}q\cdot F\big)=\frac{\partial F}{\partial \overline q}{\rm d}x,
\end{equation}
where ${\rm d}x={\rm d}x_0\wedge{\rm d}x_1\wedge{\rm d}x_2\wedge{\rm d}x_3$.

The function $F$ is said to be (\emph{left}) \emph{$\HH$-holomorphic} if
$$
\frac{\partial F}{\partial \overline q}=0.
$$
The function $F$ is said to be (\emph{left}) \emph{$\HH$-antiholomorphic}  if
$$
\frac{\partial F}{\partial q}=0.
$$
Right $\HH$-holomorphic and $\HH$-antiholomorphic functions are defined
interchanging in~\eqref{giu123} $\partial F/\partial x_\alpha$
with ${\sf i}_\alpha$ and $\overline{\sf i}_\alpha$ respectively.
For the corresponding derivative, we adopt the notation
$$
\frac{F\partial }{\partial \overline q},\qquad
\frac{F\partial }{\partial q}.
$$

For every $q_0\in\HH$, the function
$$
G(q-q_0)=\frac{\overline q-\overline q_0}{\vert q-q_0\vert^4}
$$
is left and right $\HH$-holomorphic.

The function $G(q-q_0)$ is the Cauchy-Fueter kernel and is the main ingredient to prove the basic \emph{Cauchy-Fueter formula}
$$
F(q_0)=\frac{1}{2\pi^2}\int_{q\in{\rm b\Omega}}G(q-q_o)DqF(q),
$$
where $\Omega$ is a bounded domain in $\HH$ with ${\rm b}\,\Omega$ sufficiently smooth, $q_0\in\Omega$,
and $F:\overline\Omega\to\HH$ a ${\rm C}^1$ function which is $\HH$-holomorphic in $\Omega$ and continuous on $\overline \Omega$.

From this formula and~\ref{laplaciano}, one checks immediately that left,
right $\HH$-holomorphic and $\HH$-antiholo\-morphic functions are harmonic.

For other general results in one quaternionic variable we refer to~\cite{S}.
Here we just want to mention the following ``Carleman type'' result:
\begin{pro}\label{Car}
Let $\Omega$ be a domain in the ball $B(r)=\{q\in\HH:\vert q\vert<r\}$ such
that $0\notin\overline\Omega$ and ${\rm b}\,\Omega=\Gamma\cup\Sigma$, with $\Gamma\subset B(r)$
and $\Sigma\subset{\rm b}B(r)$.
Let $F$ be an $\HH$-holomorphic function on a neighborhood of $\overline\Omega$.
Then, $F_{\vert\Omega}$ depends only on $F_{\vert\Gamma}$.
\end{pro}
\begin{proof}
Let $q\in\Omega$. By Cauchy-Fueter formula,
\[
  \begin{split}
    F(q)
    &= \frac{1}{2\pi^2}\int_{p\in{\rm b}\,\Omega}G(p-q)DpF(p) \\
    &= \frac{1}{2\pi^2}\int_{p\in\Gamma}G(p-q)DpF(p)+
       \frac{1}{2\pi^2}\int_{p\in\Sigma}G(p-q)DpF(p).
  \end{split}
\]
If $p\in\Sigma$, then $\vert q\vert<\vert p\vert$ and
$$
G(p-q)=\sum_{m=0}^{+\infty}\sum\limits_{\nu\in\sigma_m}P_\nu(q)G_\nu(p),
$$
where $\sigma_m=\{(m_1,m_2,m_3)\in\NN^3:m_1+m_2+m_3=m\}$,
the $P_\nu$ are $\HH$-holomorphic polynomials, the functions
$G_\nu(p)$ are $\HH$-holomorphic in $\HH\setminus\{0\}$, and
the series is totally convergent with respect to $p\in\Sigma$
(see \cite[Proposition~10]{S}).

Since $0\notin\overline\Omega$, by the Cauchy-Fueter theorem
(see~\cite[1.~Hauptsatz]{Fu1}) we have
$$
\int_{p\in{\rm b}\,\Omega}G_\nu(p)DpF(p)=
\int_{p\in\Gamma} G_\nu(p)DpF(p)+\int_{p\in\Sigma}G_\nu(p)DpF(p)=0,
$$
for all $\nu$. It follows that
\[
\begin{split}
\int_{p\in\Sigma}G(p-q)DpF(p) &=
\phantom{-}\sum_{m=0}^{+\infty}\sum\limits_{\nu\in\sigma_m}P_\nu(q)
\int_{p\in\Sigma}G_\nu(p)DpF(p)\\
&=-\sum_{m=0}^{+\infty}\sum\limits_{\nu\in\sigma_m}P_\nu(q)\int_{p\in\Gamma}G_\nu(p)DpF(p),
\end{split}
\]
whence the Carleman formula
$$
F(q) =
  \frac{1}{2\pi^2}\int_{p\in\Gamma}G(p-q)DpF(p)
  - \frac{1}{2\pi^2}\sum_{m=0}^{+\infty}
  \sum\limits_{\nu\in\sigma_m}P_\nu(q)\int_{p\in\Gamma} G_\nu(p)DpF(p)
$$
proving the statement.
\end{proof}

\subsection{Several variables}

Fueter operators clearly extend to ($\HH$-valued) functions of
several quaternionic variables $q_1,q_2,\dots,q_n$.

For the sake of simplicity, from now on we assume $n=2$, even
if the most part of the results proved in the sequel hold for any $n$.

We denote $q=(q_1,q_2)$ the generic element of $\HH^2$ and we set
$$
q_1=\sum_{\alpha=0}^3x_\alpha{\sf i}_\alpha,\qquad q_2=\sum_{\alpha=0}^3y_\alpha{\sf i}_\alpha.
$$
The Cauchy-Riemann-Fueter operators $\overline{\mathfrak D}$
and $\mathfrak D$ are then defined, respectively, by
\begin{equation}\label{lug2}
F\longmapsto({\partial F}/{\partial \overline q_1},{\partial F}/{\partial \overline q_2}),\qquad
F\longmapsto({\partial F}/{\partial q_1},{\partial F}/{\partial q_2})
\end{equation}
and $F$ is said to be (\emph{left}) \emph{$\HH$-holomorphic} 
if it is ${\rm C}^1$ and $\overline{\mathfrak D}F=0$.

We have the identity
\begin{equation}
\label{luB}
\begin{split}
\frac{1}{2}\big(\overline{{\rm d}q}_1\wedge{\rm d}q_1\wedge{\rm d}y&\wedge{\rm d}F+{\rm d}x\wedge\overline{{\rm d}q}_2\wedge{\rm d}{q_2}\wedge{\rm d}F\big)=\\
&-\big(\overline{{\rm D}q}_1\frac{\partial F}{\partial \overline q_1}\wedge{\rm d}y+{\rm d}x\wedge\overline{{\rm D}q}_2\frac{\partial F}{\partial \overline q_2}\big)+\star{\rm d}F,
\end{split}
\end{equation}
where ${\rm d}x={\rm d}x_0\wedge\cdots\wedge{\rm d}x_3$,
${\rm d}y={\rm d}y_0\wedge\cdots\wedge{\rm d}y_3$, and
$\star$ is the Hodge operator.

In particular, by~\ref{luB}, we get that if $F$ is $\HH$-holomorphic,
\begin{equation}\label{lu6}
\frac{1}{2}\left(
\overline{{\rm d}q}_1\wedge{\rm d}q_1\wedge{\rm d}y\wedge{\rm d}F
+{\rm d}x\wedge\overline{{\rm d}q}_2\wedge{\rm d}{q_2}\wedge{\rm d}F\right) =
\star{\rm d}F.
\end{equation}
\begin{rem}\label{lu80}
Formula~\eqref{lu6} holds, more generally,
at those points where $\overline{\mathfrak D}F=0$.
\end{rem}

Let $\Delta_1$ ($\Delta_2$) denote the laplacian in the
coordinates $x_\alpha$ ($y_\alpha$), $\alpha=0, 1, 2 ,3$.
Then, if $F$ is $\HH$-holomorphic, $\Delta_1F=\Delta_2F =0$.
In particular, $F$ is harmonic.

A useful way to construct $\HH$-holomorphic functions in
one quaternionic variable is to start by
(complex) holomorphic functions $F=F(z)=u+{\sf i}v$ and
define~\cite[5.~Satz]{Fu1}
\begin{equation}\label{cr}
 F^\#=F^\#(q):=u({\sf Re}\, q, \vert{\sf Im}\,q\vert)+
\frac{{\sf Im}\,q}{\vert {\sf Im}\,q\vert}v({\sf Re}\,q,\vert{\sf Im}\,q\vert).
\end{equation}
In general, $F^\#$ is not $\HH$-holomorphic, not even harmonic,
but its laplacian $\Delta F^\#$ is.

\begin{exa*}\label{cr0}
Let $F(z)=z^n$. Then,
$$
F^\#(q)=(z^n)^\#=q^n.
$$
In particular, for the cases $n=3$ and $n=-1$, we get
$$
\Delta q^3
= -4(2q + \bar q),
$$
$$
\Delta \left( \left( \frac{1}{z}\right)^\#\right) =
- 4 \frac{\bar q}{\lvert q\rvert^4}
= -4G(q).
$$
\end{exa*}

\subsection{Bochner-Martinelli Kernel}
The \emph{Bochner-Martinelli Kernel}
 ${\bf K}^{BM}(q,q_0)$
was introduced in~\cite{Pe1}, where a representation formula
for $\HH$-holomorphic functions was proved:
\begin{equation}\label{BM}
F(q_0)=\int_{q\in {\rm b}\,\Omega}{\bf K}^{BM}(q,q_0)F(q).
\end{equation}
Here $q_0$ belongs to a bounded domain $\Omega$ in $\HH^n$ with
smooth boundary ${\rm b}\,\Omega$ and $F$ is $\HH$-holomorphic in $\Omega$
and continuous up to ${\rm b}\,\Omega$. We will use the notation
${\bf K}^{BM}(q,q_0)$ instead of the original one. 

Set $q_1=z_1+w_1\sf j$, $q_2=z_2+w_2\sf j$, $z_\alpha,w_\alpha\in\CC$,
$z=(z_1,z_2)$, $w=(w_1,w_2)$.

We use the notation
$$
{\bf K}^{BM}(q,q_0):
={\bf K}^{BM}(z,w,z^0,w^0),
$$
where $z^0=(z^0_1,z^0_2)$, $w^0=(w^0_1,w^0_2)$.

The $\HH$-valued differential form ${\bf K}^{BM}(q,q_0)$ is a real analytic of degree~$7$ and
\begin{equation}\label{43}
{\bf K}^{BM}(q,q_0)={\bf K}^{BM}_1(q,q_0)+{\bf K}^{BM}_2(q,q_0)\sf j,
\end{equation}
where ${\bf K}_1^{BM},{\bf K}_2^{BM}$ are real analytic complex-valued differential forms.

Observe that ${\bf K}^{BM}_1(z, w, z_0, w_0)$ is the Bochner-Martinelli kernel for
functions which are holomorphic with respect to $z_1,z_2$ and antiholomorphic with respect to $w_1,w_2$ and ${\bf K}^{BM}_2(z, w, z_0, w_0)$ is exact on $\HH^2\setminus\{(z^0,w^0)\}$:
\begin{equation}
  \label{KBM2}
  {\bf K}^{BM}_2(z, w, z_0, w_0)={\rm d}\omega_2
\end{equation}
where
\begin{equation}
  \label{46}
  \begin{aligned}
    \omega_2=(8\pi^4&)^{-1}\left\lvert (z,w)-(z^0,w^0)\right\rvert^{-6}\cdot\\
    &\big({\rm d}\overline z_1\wedge{\rm d}w_1\wedge{\rm d}\overline z_2\wedge{\rm d}z_2\wedge{\rm d}\overline w_2\wedge{\rm d}w_2+\nonumber\\
    &{\rm d}\overline z_1\wedge{\rm d}z_1\wedge{\rm d}\overline w_1\wedge{\rm d}w_1\wedge{\rm d}\overline z_2\wedge{\rm d}w_2\big)\nonumber.
  \end{aligned}
\end{equation}

\subsection{\texorpdfstring{$\HH$}{H}-holomorphy and \texorpdfstring{$\HH$}{H}-convexity}

\emph{$\HH$-holomorphy} and \emph{$\HH$-convexity} are defined
like in the complex case~\cite{Pe3}.
Kontinuitätssatz holds true~\cite[Theorem 2]{Pe3},
as well as the following implications~\cite[Proposition~6, Theorem~3]{Pe3}
\begin{itemize}
\item[1)] for a domain in $\CC^4\simeq\HH^2$, holomorphy implies $\HH$-holomorphy.
The converse is not true in general (e.g.~$\HH\setminus\{(0,0)\}$ is a domain
of $\HH$-holomorphy, but it is not a domain of holomorphy in $\CC^2\simeq\HH$);
\item[2)] $\HH$-holomorphy implies $\HH$-convexity;
\end{itemize}

\noindent
For domains $\Omega\subset \HH^n$, $n>1$, with smooth boundary ${\rm b}\,\Omega$,
a necessary condition for the $\HH$-holomorphy can be given by the $2^{nd}$
fundamental form $h$ of ${\rm b}\,\Omega$ with respect to the orientation of
${\rm b}\,\Omega$ determined by the inward unit normal vector.
Precisely~\cite[Theorem 4]{Pe3},
\begin{itemize}
\item[3)] given a point $q_0\in{\rm b}\,\Omega$, there is no right $\HH$-line
$\ell$ tangent to ${\rm b}\,\Omega$ at $q_0$ such that $h(q_0)|_\ell<0$.
\end{itemize}
In this case, we say that $\Omega$ (or its boundary) is \emph{Levi $\HH$-convex}.
For $n=2$,
we say that $\Omega$ is \emph{strongly Levi $\HH$-convex}, if for
all $q_0 \in {\rm b}\,\Omega$, we have $h(q_0)|_\ell > 0$, where $\ell$ is
the only right $\HH$-line tangent to ${\rm b}\,\Omega$ at $q_0$.

In general, we say that a smooth hypersurface $S\subset \HH^n$
is \emph{nondegenerate} if,
there exists a right $\HH$-line $\ell$ such that the form
$h(q_0)|_\ell$ has constant sign.

Two open problems:
\begin{itemize}
\item[i)] Is a domain $\HH$-convex a domain of $\HH$-holomorphy?
\item[ii)] Levi problem in $\HH^n$.
\end{itemize}

\subsection{\texorpdfstring{$\overline{\mathfrak D}$}{barD}-problem and Hartogs Theorem}\label{13ott}

Let $q=(q_1,q_2)\in\HH^2$ with
$$
q_1=\sum_{\alpha=0}^3x_\alpha i_\alpha,\qquad q_2=\sum_{\alpha=0}^3y_\alpha i_\alpha
$$
and consider the laplacians
$$
\Delta_1 = \frac{\partial^2}{\partial x_0^2}+\frac{\partial^2}{\partial x_1^2}+\frac{\partial^2}{\partial x_2^2}+\frac{\partial^2}{\partial x_3^2},\qquad
\Delta_2 = \frac{\partial^2}{\partial y_0^2}+\frac{\partial^2}{\partial y_1^2}+\frac{\partial^2}{\partial y_2^2}+\frac{\partial^2}{\partial y_3^2}.
$$
Then, since ${\partial}/{\partial \overline q_s}$ and $\Delta_h$ commute we have
 \begin{equation}\label{cr00}
\frac{\partial}{\partial \overline q_s}\frac{\partial}{\partial q_h}\frac{\partial}{\partial \overline q_h}
=\frac{\partial}{\partial \overline q_s}{\Delta}_h={\Delta}_h\frac{\partial}{\partial \overline q_s}.
\end{equation}
It follows that, if $u$ is a smooth (local) solution of the
{\rm CRF} \emph{system}
\begin{equation}
\overline{\mathfrak D} u=g,\quad g=(g_1,g_2),
\end{equation}
then
\begin{equation}\label{cr2}
{\Delta}_h g_s = \frac{\partial}{\partial \overline q_s}\frac{\partial g_h}{\partial q_h},
\end{equation}
which is a nontrivial condition for $h\neq s$.

For every pair $g=(g_1,g_2)$, we set
\begin{equation}\label{cr02}
  \begin{aligned}
    \overline P_1(g)&= \frac{\partial}{\partial \overline q_1}\frac{\partial g_2}{\partial q_2}-\Delta_2 g_1,\\
    \overline P_2(g)&= \frac{\partial}{\partial \overline q_2}\frac{\partial g_1}{\partial q_1}-\Delta_1 g_2
  \end{aligned}
\end{equation}
and denote $\overline P$ the operator $g=(g_1,g_2)\mapsto(\overline P_1(g),\overline P_2(g))$.
Then, if $g=\overline{\mathfrak D}u$ with $u$ smooth, we have
\begin{equation}\label{cr3}
\overline P(g)=0,
\end{equation}
i.e.,
\begin{equation}\label{cr4}
\overline P_1(g)=0,\qquad
\overline P_2(g)=0.\\
\end{equation}
Conditions~\eqref{cr2} for $h,s=1,\dots,n$ are still necessary in order to
solve $\overline{\mathfrak D} u=g$ for $g=(g_1,\dots,g_n)$.
If $g\in C_0^k$, $n, k\ge 2$, they are also sufficient and in such
situation $\overline{\mathfrak D} u=g$ has a $C_0^k$ solution $u$
(see \cite[Theorem 1]{Pe2}).
In particular, this implies Hartogs Theorem.
We point out that Hartogs Theorem was already proved by the second
author~\cite[Teorema 6]{Pe1}, by solving the equation $\overline{\mathfrak D} u=g$
with integral conditions on $g$, instead of~\eqref{cr2}.
As for the system $\overline{\mathfrak D} u=g$, when $g\in C^\infty(\Omega,\HH)$,
$\Omega\subset\HH^n$, we have the following: if $n=2$ and $\Omega$ is convex, the system
has a smooth solution if and only if $\overline P(g)=0$ (see~\cite{ABLSS}).
If $n>2$, conditions~\ref{cr2} are no longer sufficient in general.
For $g\in C^\infty(\Omega,\HH)$, $\Omega$ convex, using the results of~\cite{AL,CSSS},
necessary and sufficient conditions were proved in~\cite{BDS}.
\begin{rem}\label{14ott}
The same is true if $g$ is replaced by a distribution.
This is a consequence of the ``division of distributions''~\cite{Eh,Ma,Pa,AN,N}.
We will use this generalization in Section~\ref{F1}.
\end{rem}
As far as we know, nothing is known about the existence of the equation $\overline{\mathfrak D} u=g$ in more general domains.  
\section{Riemann-Hilbert and Dirichlet problems for \texorpdfstring{$\HH$}{H}-holomorphic functions.}
\label{TEXI}

\subsection{The operator \texorpdfstring{${\overline{\mathfrak D}}_{\rm b}$}{barDb} and the CRF condition}
\label{sec:crf}

Let $\Omega\subset \HH^2$ be a domain.
A \emph{domain splitting} $(S,U^+,U^-)$ of $\Omega$ is given by a smooth (nonempty)
hypersurface $S$ closed in $\Omega$ and two open disjoint nonempty sets $U^+$, $U^-$,
such that $\Omega\setminus S = U^+ \cup U^-$, where
both $U^+$ and $U^-$ have boundary $S$ in $\Omega$.

We say that a continuous (resp.~smooth\footnote{For convenience of exposition,
since our work reposes in an essential way to the theory of Ehrenpreis and its
applications~\cite{Eh,CSSS}, we restrict ourselves to the class of $C^\infty$
functions, even if some definitions and constructions can be given in a more
general setting.}) function $f:S\to\HH$ is
a \emph{continuous (resp.~smooth) jump} relative to a domain
splitting $(S,U^+,U^-)$ of $\Omega$, if there exist two $\HH$-holomorphic
functions $F^+$, $F^-$, on $U^+$, $U^-$ respectively, such that
$F^+$, $F^-$ are continuous (resp.~smooth) up to $S$ and $f=F^+|_S-F^-|_S$.

A function $f:S\to\HH$ (continuous or smooth) is \emph{locally a jump} if,
for every $q_0\in S$, there exists a neighborhood $U$ of $q_0$
such that $f|_{U\cap S}$ is a jump in $U$.

Observe that the functions $F^+$, $F^-$ are determined up
an $\HH$-holomorphic function in $U$.
In particular, if $S$ is the boundary of a bounded domain in $\HH^2$, 
Dirichlet problem reduces to Riemann-Hilbert problem via the Hartogs theorem.

Both these problems require conditions on the given function $f: S\to\HH$
that we call CRF conditions.

Let $S$ be defined by $\rho=0$. We say that a smooth function $f:S\to\HH$
is a (\emph{left}) \emph{CRF function} if,
there is a smooth extension $F$ of $f$ on a neighborhood of $S$,
such that we have
\begin{equation}\label{tan1}
\overline{\mathfrak D} F=\rho\cdot A+\overline{\mathfrak D}\rho \cdot B,
\end{equation}
with $A$ and $B$ smooth.
The {\rm CRF} condition is independent of the extension $F$,
as well as of the equation of $S$.

The {\rm CRF} condition can be given in a more intrinsic way,
as shown in Theorem~\ref{lu56} below.

\begin{rem}
\label{rmk:smooth-extension}
Observe that, $f$ is a {\rm CRF} function if and only if there exists a smooth
extension $F_1$ of $f$ with $\overline{\mathfrak D} F_1=0$ on $S$.
(It is enough to take $F_1=F-\rho\cdot B$, where $F$ satisfies~\ref{tan1}.)
\end{rem}

Clearly, if $F$ is an $\HH$-holomorphic function on one sided neighborhood of $S$,
then $F|_S$ is a {\rm CRF} function, in particular, every local jump $f$ 
on $S$ is a {\rm CRF} function.

We will see below that, unlike the complex case, trace conditions
on $f$ involve both first-order and second-order differential
equations (Remark~\ref{rmk:restrict-crf}).

This is not surprising, due to the fact that Riemann-Hilbert problem
is related to local solvability of $\overline{\mathfrak D}u=g$
and this requires a second-order differential condition for $g$.

If $F=U+V{\sf j}$ is an extension of
$f$, $q_1=z_1+w_1\sf j$, $q_2=z_2+w_2\sf j$,
where $U,V,z_1,w_1,z_2,w_2$ are complex,
then the {\rm CRF} condition writes
\begin{equation}\label{tan2}
{\rm rank} \left(\begin{array}{ccccc}
U_{\overline z_1}-\overline V_{\overline w_1 }&\rho_{\overline z_1}&-\rho_{\overline w_1}\\
\overline V_{z_1}+U_{w_1}&\rho_{w_1}& \rho_{z_1}\\
U_{\overline z_2}-\overline V_{\overline w_2} &\rho_{\overline z_2}&-\rho_{\overline w_2}\\
\overline V_{z_2}+U_{w_2}&\rho_{w_2}& \rho_{z_2}
\end{array}\right)<3.
\end{equation}

\subsubsection{{\rm CRF} condition and extendability}
Suppose $S$ oriented.  Denote $\omega$ the volume form
of $S$ and $\nu=(\nu_1,\nu_2)$, $\nu_1,\nu_2\in\HH$,
the unit normal vector which gives the orientation of $S$.

Let $\langle,\rangle:\HH^2\times\HH^2\to\HH$ be the scalar product
$$
\big\langle (q_1,q_2), (p_1,p_2 )\big\rangle=\overline q_1p_1+\overline q_2 p_2.
$$
By direct computation, one verify that
\begin{equation}\label{lu50}
\big(\overline{Dq}_1\wedge{\rm d}y\big)\big|_S=-\overline\nu_1\omega,\qquad
\big({\rm d}x\wedge\overline{Dq}_2\big)\big|_S=-\overline\nu_2\,\omega.
\end{equation}
Let $f:S\to\HH$ be smooth and $F$ a smooth extension of $f$ on a neighborhood of $S$.
Then, by restriction to $S$, from~\eqref{luB} we get

\begin{equation}\label{lu51}
-\frac{1}{2}\Big(\overline{{\rm d}q}_1\wedge{\rm d}q_1\wedge{\rm d}y
+{\rm d}x\wedge\overline{{\rm d}q}_2\wedge{\rm d}{q_2}\Big)\Big|_S\wedge{\rm d}f
= \Big(-\big\langle \nu,{\overline{\mathfrak D}F|_S}\big\rangle+\frac{\partial F}{\partial\nu}\Big)\omega,
\end{equation}
where $\overline{\mathfrak D}F=\big(\frac{\partial F}{\partial\overline q_1},\frac{\partial F}{\partial\overline q_2}\big)$.

Let $f^\perp:S\to\HH$ be the smooth function defined by
\begin{equation}\label{lu52}
-\frac{1}{2}\Big(\overline{{\rm d}q}_1\wedge{\rm d}q_1\wedge{\rm d}y
+{\rm d}x\wedge\overline{{\rm d}q}_2\wedge{\rm d}{q_2}\Big)\Big|_S\wedge{\rm d}f
= f^\perp\cdot\omega
\end{equation}
and set
\begin{equation}\label{lu53}
\frac{\partial}{\partial x_\alpha}\Big|_S
=\tau_{x_\alpha}+\left(\frac{\partial}{\partial x_\alpha},\nu\right)\nu,\qquad
\frac{\partial}{\partial y_\alpha}\Big|_S
=\tau_{y_\alpha}+\left(\frac{\partial}{\partial y_\alpha},\nu\right)\nu
\end{equation}
$\alpha=0, 1, 2, 3$, where $(\cdot,\cdot)$ denotes the euclidean scalar
product of $\RR^8$ and $\tau_{x_\alpha}$, $\tau_{y_\alpha}$ are the tangential
components of $\frac{\partial}{\partial x_\alpha}\big|_S$, $\frac{\partial}{\partial y_\alpha}\big|_S$ respectively.

We set
\begin{equation}
  \label{lu54}
  \begin{aligned}
  f_{(x_\alpha)}&=\tau_{x_\alpha}(f)+\big(\frac{\partial}{\partial x_\alpha},\nu\big)f^\perp,\\
  f_{(y_\alpha)}&=\tau_{y_\alpha}(f)+\big(\frac{\partial}{\partial y_\alpha},\nu\big)f^\perp,\\
  f_{(\overline q_1)}&=f_{(x_0)}+{\sf i}f_{(x_1)}+{\sf j}f_{(x_2)}+{\sf k}f_{(x_3)},\\
  f_{(\overline q_2)}&=f_{(y_0)}+{\sf i}f_{(y_1)}+{\sf j}f_{(y_2)}+{\sf k}f_{(y_3)};
  \end{aligned}
\end{equation}
they are smooth functions on $S$.

\begin{pro}\label{lu12}
Let $f:S\to\HH$ be a smooth {\rm CRF} function and $F$ a smooth
local extension of $f$ such that $\overline{\mathfrak D}F=0$ on $S$.
Then,
\begin{equation}
\label{eq:s25}
\frac{\partial F}{\partial \nu} = f^\perp,\quad
\frac{\partial F}{\partial {x_\alpha}}\Big|_S = f_{(x_\alpha)},\quad
\frac{\partial F}{\partial {y_\alpha}}\Big|_S=f_{(y_\alpha)},
\end{equation}
for $\alpha=0, 1, 2, 3$.
\end{pro}
\begin{proof}
Since $\overline{\mathfrak D}F=0$ on $S$ and $F|_S=f$, by Remark~\ref{lu51}
$$
-\frac{1}{2}\left(\overline{{\rm d}q}_1\wedge{\rm d}q_1\wedge{\rm d}y\wedge{\rm d}f
+{\rm d}x\wedge\overline{{\rm d}q}_2\wedge{\rm d}{q_2}\wedge{\rm d}f\right)\big|_S
=\frac{\partial F}{\partial \nu}\omega
$$
and comparing with~\eqref{lu52} we then have $\frac{\partial F}{\partial \nu}=f^\perp$.
Formulas~\eqref{lu53} now
imply $\frac{\partial F}{\partial {x_\alpha}}\big|_S=f_{(x_\alpha)}$,
$\frac{\partial F}{\partial {y_\alpha}}\big|_S=f_{(y_\alpha)}$, $\alpha=0, 1, 2, 3$.
\end{proof}

\begin{rem}
\label{rmk:restrict-crf}
If $f$ is the boundary value of an $\HH$-holomorphic function $F$,
then, by Proposition~\ref{lu12}, we get
\[
  \label{eq:lu12}
  \begin{aligned}
    {\frac{\partial F}{\partial x_\alpha}}\big|_S&=f_{(x_\alpha)}\\
    {\frac{\partial F}{\partial y_\alpha}}\big|_S&=f_{(y_\alpha)}
  \end{aligned}
  \qquad\text{for $\alpha=0,1, 2, 3$.}
\]
Since the operators $\overline{\mathfrak D}$,
$\partial/\partial x_\alpha$, $\partial/\partial y_\alpha$ commute,
$f_{(x_\alpha)}$ and $f_{(y_\alpha)}$ are restrictions of
the $\HH$-holomorphic functions $\frac{\partial F}{\partial x_\alpha}$
and $\frac{\partial F}{\partial y_\alpha}$ respectively,
hence $f_{(x_\alpha)}$, $f_{(y_\alpha)}$ are CRF functions too.
\end{rem}

A smooth {\rm CRF} function $f:S\to\HH$ is said to be \emph{admissible}
if $f_{(x_\alpha)},f_{(y_\alpha)}$, $\alpha=0,1, 2, 3$, are ${\rm CRF}$ functions too.
Unlike the complex case, a {\rm CRF} function is not admissible in general.
Here is a counterexample:
\begin{exa*}
Let $S=\{y_3=0\}$, $f=-x_1y_0{\sf j}+x_0y_0{\sf k}$.
Since $\partial f/\partial \overline q_1=0$, $f$ is {\rm CRF}.
Moreover, $f^\perp=f_{(y_3)}=-x_0+x_1{\sf i}$.
In particular, if $f_{(y_3)}$ were CRF 
we should have $\partial f_{(y_3)}/\partial \overline q_1=0$,
whereas $\partial f_{(y_3)}/\partial \overline q_1=-2$.
\end{exa*}

\subsubsection{The tangential operator ${\overline{\mathfrak D}}_{\rm b}$}

The {\rm CRF} condition determines a differential operator on $S$ that
will be denoted by ${\overline{\mathfrak D}}_{\rm b}$.
We want to write explicitly the operator ${\overline{\mathfrak D}}_{\rm b}$.

Consider on $S$ the following $\HH$-valued differential forms
\begin{equation}
  \label{lu62}
  \begin{aligned}
    {\rm d}_{(q_1)}f
    &= f_{(x_0)}{\rm d}{x_0}|_S+f_{(x_1)}{\rm d}{x_1}|_S+f_{(x_2)}{\rm d}{x_2}|_S+f_{(x_3)}{\rm d}{x_3}|_S\\
    {\rm d}_{(q_2)}f
    &=f_{(y_0)}{\rm d}{y_0}|_S+f_{(y_1)}{\rm d}{y_1}|_S+f_{(y_2)}{\rm d}{y_2}|_S+f_{(y_3)}{\rm d}{y_3}|_S.
  \end{aligned}
\end{equation}
The following equalities hold
\begin{equation}
  \label{lu55}
  \begin{aligned}
    {{\rm D}q_1}|_S\wedge{\rm d}_{(q_1)}f
    &=-f_{(\overline q_1)}\,{{\rm d}x}|_S\\
    {{\rm D}q_2}|_S\wedge{\rm d}_{(q_2)}f
    &=-f_{(\overline q_2)}\,{{\rm d}y}|_S.
  \end{aligned}
\end{equation}
We have the following
\begin{teo}\label{lu56}
For a given smooth function $f$ on $S$ the following conditions are equivalent:
\begin{itemize}
\item[a)] $f$ is a {\rm CRF} function;
\item[b)] $f_{(\overline q_1)} \equiv f_{(\overline q_2)} \equiv 0$;
\item[c)] ${{\rm D}q_1}|_S\wedge{\rm d}_{(q_1)}f \equiv {{\rm D}q_2}|_S\wedge{\rm d}_{(q_2)}f \equiv 0$.
\end{itemize}
\end{teo}
\begin{proof}
Let $f$ be {\rm CRF}.
Then,
there exists a smooth extension $F$ of $f$ with
the property $\overline{\mathfrak D}F=0$ on $S$
(Remark~\ref{rmk:smooth-extension}).
From~\eqref{eq:s25}, we get
$$
  \partial F/\partial\nu=f^\perp,\qquad
  \frac{\partial F}{\partial x_\alpha}\Big|_S
  =f_{(x_\alpha)},\qquad
  \frac{\partial F}{\partial y_\alpha}|_S=f_{(y_\alpha)},
$$
$\alpha=0, 1, 2, 3.$
Consequently
\begin{equation}\label{lu57}
\frac{\partial F}{\partial \overline q_1}\Big|_S
=f_{(\overline q_1)},\qquad \frac{\partial F}{\partial \overline q_2}\Big|_S
=f_{(\overline q_2)}.
\end{equation}
By hypothesis, $\overline{\mathfrak D}F=0$ on $S$
hence $f_{(\overline q_1)}=f_{(\overline q_2)}=0$ and therefore, by~\eqref{lu55}
\begin{equation}\label{lu58}
{{\rm D}q_1}|_S\wedge{\rm d}_{(q_1)}f={{\rm D}q_2}|_S\wedge{\rm d}_{(q_2)}f=0,
\end{equation}
i.e.,~b) and~c).

Assume that $f$ satisfies c).
Then, by~\eqref{lu55}, we have $f_{(\overline q_1)}\,{\rm d}x|_S\equiv 0$,
$f_{(\overline q_2)}\,{\rm d}y|_S\equiv 0$ and,
if ${\rm d}x|_S(p)\neq 0$, ${\rm d}y|_S(p)\neq 0$, $p\in S$,
then $f_{(\overline q_1)}(p)=f_{(\overline q_2)}(p)=0$.
Suppose, for instance, that ${\rm d}x|_S(p)=0$.
Then, the second of~\eqref{lu50} implies $\nu_2(p)=0$,
i.e., $\nu(p)=\big(\nu_1(p),0\big)$, where $\nu_1(p)\neq 0$.
Thus, ${\rm d}y|_S(p)\neq 0$, otherwise (again by~\eqref{lu50}),
we should have $\nu_1(p)=0$,
consequently, $f_{(\overline q_2)}(p)=0$.

Let us show that necessarily $f_{(\overline q_1)}(p)=0$.
By standard argument of differential topology, it is easy to construct a
smooth extension $F$ of $f$ such that $\partial F/\partial\nu=f^\perp$.
Identity~\eqref{lu51} and definition of $f^\perp$ then imply
that $\langle \nu(p),\overline {\mathfrak D} F(p)\rangle=0$, i.e.,
$\overline\nu_1(p)\partial F/\partial \overline q_1(p)=0$, whence $\partial F/\partial \overline q_1(p)=0$.
Arguing as in the first part of the proof,
we get $\partial F/\partial {\overline q_1}|_S=f_{(\overline q_1)}$,
$\partial F/\partial {\overline q_2}|_S=f_{(\overline q_2)}$, in
particular, also $f_{(\overline q_1)}(p)=0$
for every $p\in S$, and c) imply b).
Furthermore, $\overline{\mathfrak D} F = 0$ on $S$,
hence c) implies a) too.
\end{proof}
We denote ${\overline{\mathfrak D}}_{\rm b}$ the operator
\begin{equation}\label{lu60}
  {\overline{\mathfrak D}}_{\rm b} : 
f\longmapsto
\Big({{\rm D}q_1}|_S\wedge {\rm d}_{(q_1)}f
,{{\rm D}q_2}|_S\wedge {\rm d}_{(q_2)}f\Big).
\end{equation}

\subsection{Solvability of the Riemann-Hilbert problem}
We want to prove that for smooth admissible functions the local
Riemann-Hilbert problem is always solvable.

We consider an orientable smooth hypersurface $S$,
given as the zero set of a smooth
function $\rho$ such that $\nabla\rho\neq 0$ around $S$.
\begin{pro}\label{ag11}
Let $f:S\to\HH$ be a smooth function. The following properties are equivalent
\begin{itemize}
\item[i)] $f$ is admissible;
\item[ii)] there exists a smooth extension $F$ of $f$ such that
around $S$ one has $\overline{\mathfrak D}F=\rho^2u$,
with $u$ a smooth $\HH^2$-valued map.
\end{itemize}
\end{pro}
\begin{proof}
Let $F$ as in ii). Clearly $f$ is {\rm CRF}.
By Proposition~\ref{lu12}, we have $\partial F/{\partial x_\alpha}|_S=f_{(x_\alpha)}$,
$\partial F/{\partial y_\alpha}|_S=f_{(y_\alpha)}$, $\alpha=0, 1, 2, 3$. Moreover,
\begin{equation}
  \begin{aligned}
    \overline{\mathfrak D}\Big(\frac{\partial F}{\partial x_\alpha}\Big)
    &=\frac{\partial}{\partial x_\alpha}\Big(\overline{\mathfrak D} F\Big)
    =\rho\Big(2\frac{\partial\rho}{\partial x_\alpha}u+\rho\frac{\partial u}{\partial x_\alpha}\Big)\\
    \overline{\mathfrak D}\Big(\frac{\partial F}{\partial y_\alpha}\Big)
    &=\frac{\partial}{\partial y_\alpha}\Big(\overline{\mathfrak D} F\Big)
    =\rho\Big(2\frac{\partial\rho}{\partial y_\alpha}u+\rho\frac{\partial u}{\partial y_\alpha}\Big),
  \end{aligned}
  \qquad \text{for $\alpha=0, 1, 2, 3$.}
\end{equation}
Therefore, $\partial F/{\partial x_\alpha}$ ($\partial F/{\partial y_\alpha}$), $\alpha=0, 1, 2, 3$,
is a smooth extension of $f_{(x_\alpha)}$ ($f_{(y_\alpha)}$),
whose $\overline{\mathfrak D}$ is vanishing on $S$.
It follows that $f$ is admissible.

Assume now that $f$ is admissible, in particular {\rm CRF}.
Therefore, there is a smooth extension $G$ of $f$ and a
smooth $\HH^2$-valued map $\psi$ such that $\overline{\mathfrak D} G=\rho\psi$.
Again, by Proposition~\ref{lu12}, one has $\partial G/{\partial x_\alpha}|_S=f_{(x_\alpha)}$,
$\partial G/{\partial y_\alpha}|_S=f_{(y_\alpha)}$, $\alpha=0, 1, 2, 3$.
Since also $f_{(x_\alpha)}$ is {\rm CRF}, there is a smooth
extension $F^{(x_\alpha)}$ of $f_{(x_\alpha)}$ such
that $\overline{\mathfrak D}F^{(x_\alpha)}=\rho\eta^{(x_\alpha)}$
with $\eta^{(x_\alpha)}$ smooth, whence
\begin{equation}\label{ag12}
\partial G/{\partial x_\alpha}-F^{(x_\alpha)}=\rho\psi^{(x_\alpha)}
\end{equation}
with $\psi^{(x_\alpha)}$ smooth, $\alpha=0, 1, 2, 3$.

Applying $\overline{\mathfrak D}$ to~\ref{ag12}, and taking into account
that $\overline{\mathfrak D}\circ(\partial/\partial x_{\alpha})=(\partial/\partial x_{\alpha})\circ\overline{\mathfrak D}$,
we obtain
\begin{equation}\label{ag13}
\frac{\partial(\overline{\mathfrak D}G)}{\partial x_{\alpha}}
=\rho H^{(x_\alpha)}+ \overline{\mathfrak D}\rho\cdot\psi^{(x_\alpha)}
\end{equation}
with $H^{(x_\alpha)}$ smooth, $\alpha=0, 1, 2, 3$.

In the same way,
\begin{equation}\label{ag14}
\frac{\partial(\overline{\mathfrak D}G)}{\partial y_{\alpha}}
=\rho H^{(y_\alpha)}+ \overline{\mathfrak D}\rho\cdot\psi^{(y_\alpha)}
\end{equation}
with $H^{(y_\alpha)}$ smooth, $\alpha=0, 1, 2, 3$.

Let
$$
\nu=\nabla \rho/\vert\nabla \rho\vert=\vert\nabla \rho\vert^{-1}\sum_{\alpha=0}^3\Big(\frac{\partial\rho}{\partial x_{\alpha}}\frac{\partial}{\partial x_{\alpha}}+\frac{\partial\rho}{\partial y_{\alpha}}\frac{\partial}{\partial y_{\alpha}}\Big).
$$

By hypothesis, $\overline{\mathfrak D} G=\rho\psi$, so
\begin{equation}\label{ag15}
\frac{\partial(\overline{\mathfrak D}G)}{\partial \nu}
=\frac{\partial\rho}{\partial \nu}\psi+\rho\frac{\partial\psi}{\partial\nu}=\vert\nabla\rho\vert\psi+\rho\frac{\partial\psi}{\partial\nu}
\end{equation}
On the other hand, from~\eqref{ag13},~\eqref{ag14}, we derive
\begin{equation}\label{ag16}
\frac{\partial(\overline{\mathfrak D}G)}{\partial \nu}
=\vert\nabla \rho\vert^{-1}\Big\{\rho\sum_{\alpha=0}^3 A_\alpha+\overline{\mathfrak D}\rho\cdot\sum_{\alpha=0}^3 B_\alpha\Big\}.
\end{equation}
Equalizing~\eqref{ag15} and~\eqref{ag16}, we get
\begin{equation}\label{ag17}
\psi=\rho\Phi+2\overline{\mathfrak D}\rho\cdot\Theta
\end{equation}
and consequently
$$
\overline{\mathfrak D}G=\rho\psi=\rho^2\Phi+\overline{\mathfrak D}\rho^2\cdot\Theta.
$$
Then, $F:=G-\rho^2\Theta$ is the desired extension of $f$.
\end{proof}

\begin{lem}\label{ag18}
Let $U$ be a domain in $\HH^2$, $S=\{\rho=0\}$ where $\rho:U\to\HH$ is
smooth and $\nabla\rho\neq0$ on $S$.
Let $\{h_k\}_k\in\NN$ be a sequence of smooth functions $S\to\HH$.
Then, there exists a smooth function $E:U\to\HH$ with the following properties
\begin{itemize}
\item[1)] $E|_S=h_0$;
\item[2)] $\frac{\partial^k E}{\partial\rho^k}\big|_S=h_k$, \ $\forall k\ge 1$.
\end{itemize}
\end{lem}
This lemma is a straightforward generalization of~\cite[Proposition~2.2]{AnH}.
\begin{pro}\label{ag19}
Let $U$ be a domain in $\HH^2$, $S=\{\rho=0\}$ where $\rho:U\to\RR$ is
smooth and $\nabla\rho\neq0$ on $S$.
Let $f:S\to\HH$ be a smooth and admissible function.
Then there are a smooth function $F:U\to\HH$ and two
sequences $\{\alpha_k\}_{k\ge 2}$, $\{\beta_k\}_{k\ge 2}$ of smooth
functions $U\to\HH$ and $U\to\HH^2$, respectively,
satisfying the following conditions:
\begin{itemize}
\item[1)] $F|_S=f$;
\item[2)] $({\partial^k\alpha_m}/{\partial\rho^k})|_S=0$, $\forall k\ge1,m\ge2$;
\item[3)] $\overline{\mathfrak D}\big(F-\sum_{k=2}^m(\rho^k/k)\alpha_k\big)
=\rho^m\beta_m$, \ $\forall m\ge 2$.
\end{itemize}
\end{pro}
\begin{proof}
Since $f$ is admissible, by Proposition~\ref{ag11} there is a smooth
extension $F:U\to\HH$ of $f$ such that $\overline{\mathfrak D}F=\rho^2\sigma$.
We construct the sequences by recurrence assuming $\alpha_2=0$, $\beta_2=\sigma$
in such a way second and third conditions of the proposition are
satisfied for $m=2$.

Suppose that $\alpha_2,\ldots,\alpha_m$, $\beta_2,\ldots,\beta_m$ are already constructed
in such a way that the above conditions~2) and~3) are satisfied for
all integers $s\le m$, $k\ge 1$,
in order to define $\alpha_{m+1}$ and $\beta_{m+1}$.

Set
$$
G=F-\sum_{k=2}^m(\rho^k/k)\alpha_k, \qquad
\beta_m=(\zeta_1,\zeta_2)\in\HH^2.
$$
By definition, $\partial G/\partial\overline q_h=\rho^m\zeta_h$, $h=1, 2$,
hence $\overline{\mathfrak D} G$ satisfies the condition~\eqref{cr2}, that is
$$
\frac{\partial}{\partial \overline q_s}\frac{\partial}{\partial q_s}\frac{\partial G}{\partial \overline q_h}=
\frac{\partial}{\partial \overline q_h}\frac{\partial }{\partial q_s}\frac{\partial G}{\partial \overline q_s},
$$
which gives
$$
\frac{\partial}{\partial \overline q_s}\Big(m\rho^{m-1}\frac{\partial\rho}{\partial q_s}\zeta_h+\rho^m\frac{\partial\zeta_h}{\partial q_s}\Big)=
\frac{\partial}{\partial \overline q_h}\Big(m\rho^{m-1}\frac{\partial\rho}{\partial q_s}\zeta_s+\rho^m\frac{\partial\zeta_s}{\partial q_s}\Big).
$$
Taking into account that $\rho$ is real and $m\ge2$ we get
\begin{equation}
  \label{ag20}
  \begin{aligned}
  &m(m-1)\rho^{m-2}\frac{\partial\rho}{\partial\overline q_s}\frac{\partial\rho}{\partial q_s}
  \zeta_h+m\rho^{m-1}\frac{\partial}{\partial\overline q_s}\Big(\frac{\partial\rho}{\partial q_s}\zeta_h
  \Big)+m\rho^{m-1}\frac{\partial\rho}{\partial\overline q_s}\frac{\partial\zeta_h}{\partial q_s}+
  \rho^m\frac{\partial}{\partial\overline q_s}\Big(\frac{\partial\zeta_h}{\partial q_s}\Big)
  =\\
  &m(m-1)\rho^{m-2}\frac{\partial\rho}{\partial\overline q_h}\frac{\partial\rho}{\partial q_s}\zeta_s+
  m\rho^{m-1}\frac{\partial}{\partial\overline q_h}\Big(\frac{\partial\rho}{\partial q_s}\zeta_s\Big)+
  m\rho^{m-1}\frac{\partial\rho}{\partial\overline q_h}\frac{\partial\zeta_s}{\partial q_s}
  +\rho^m\frac{\partial}{\partial\overline q_h}\Big(\frac{\partial\zeta_s}{\partial q_s}\Big).
  \end{aligned}
\end{equation}
Summing with respect to $s$ the above equalities and
dividing by $m(m-1)$, for fixed $h=1, 2$ we get
\begin{equation}\label{ag21a}
\rho^{m-2}\vert\nabla\rho\vert^2\zeta_h
=\rho^{m-2}\frac{\partial\rho}{\partial\overline q_h}\Big(\sum_{s=1}^2\frac{\partial\rho}{\partial q_s}\zeta_s\Big)+
\rho^{m-1}l_h
\end{equation}
with $l_h\in{\rm C}^\infty(U)$
$h=1, 2$
whence
\begin{equation}\label{ag21b}
\vert\nabla\rho\vert^2\zeta_h
=\frac{\partial\rho}{\partial\overline q_h}\Big(\sum_{s=1}^2\frac{\partial\rho}{\partial q_s}\zeta_s\Big)+
\rho l_h
\end{equation}
$h=1, 2$.
Since $\nabla\rho\neq 0$ on $S$, on an open neighborhood $V\subset U$ of $S$ we have
$$
\zeta_h=\frac{\partial\rho}{\partial\overline q_h}g+\rho\gamma_h
$$
$h=1, 2$, with $g,\gamma_h\in{\rm C}^\infty(V)$.

Setting $\gamma=(\gamma_1,\gamma_2)$, and recalling that $\beta_m=(\zeta_1,\zeta_2)$, on $V$
we have $\beta_m=(\overline{\mathfrak D}\rho) g+\rho\gamma$,
so, by the beginning assumption, we derive
\begin{equation}
  \label{ag22}
  \begin{split}
  \overline{\mathfrak D}\Big(F-\sum_{k=2}^m(\rho^k/k)\alpha_k\Big)
  &= \rho^m\beta_m=\rho^m\overline{\mathfrak D}\rho\cdot g+\rho^{m+1}\gamma\\
  &= \overline{\mathfrak D}\big(\rho^{m+1}g/(m+1)\big)-
     \frac{\rho^{m+1}}{m+1}\overline{\mathfrak D}g+\rho^{m+1}\gamma.
  \end{split}
\end{equation}
With
$$
\theta=\gamma-\overline{\mathfrak D} g/(m+1)
$$
equation~\eqref{ag22} rewrites
\begin{equation}\label{ag23}
\overline{\mathfrak D}\Big(F-\sum_{k=2}^m(\rho^k/k)\alpha_k-\rho^{m+1}g/(m+1)\Big)=\rho^{m+1}\theta.
\end{equation}
Observe that $g$ and $\theta$ can be chosen in such a way that an
equality like~\eqref{ag23} holds on $U$.
(It is enough to consider a closed neighborhood $V'\subset V$ of $S$, a smooth
extension of $g|_{V'}$ to $U$ and take $\theta$ according to~\eqref{ag23}).

By Lemma~\ref{ag18}, there exists a smooth function $\alpha_{m+1} : U\to\HH$, such
that $\alpha_{m+1}|_S=g|_S$, ${\partial^k\alpha_{m+1}}/{\partial \rho^k}|_S=0$ for
every $k\ge1$.
Then, $\alpha_{m+1}-g=\rho\varepsilon$ with $\varepsilon:U\to\HH$ smooth and, consequently,
\[
\begin{split}
-\overline{\mathfrak D}\Big(\frac{\rho^{m+1}g}{m+1}\Big)
&= \overline{\mathfrak D}\Big(\frac{\rho^{m+2}\varepsilon}{m+1}\Big)
-\overline{\mathfrak D}\Big({\frac{\rho^{m+1}}{m+1}}\alpha_{m+1}\Big) =\\
&
-\overline{\mathfrak D}\Big({\frac{\rho^{m+1}}{m+1}}\alpha_{m+1}\Big) +
  \frac{\rho^{m+1}}{m+1}\left((m+2)\overline{\mathfrak D}\rho\cdot\varepsilon
  +\rho\overline{\mathfrak D}\varepsilon\right).
\end{split}
\]
If we define
$$
\zeta_{m+1}=
\theta -\frac{1}{m+1}\Big((m+2)\overline{\mathfrak D}\rho\cdot\varepsilon+\rho\overline{\mathfrak D}\varepsilon\Big)
$$
$\alpha_{m+1}$ and $\zeta_{m+1}$ satisfy conditions 2) and 3) of the proposition for $m+1$.
\end{proof}

Let $U$, $\rho$, $S$ be as in Proposition~\ref{ag19} and
let $G:U\to\HH^r$ be a smooth map.
We say that $G$ \emph{vanishes of infinite order} on $S$ or
that $G$ is \emph{flat} on $S$ if, for any integer $k$,
$$
\lim\limits_{\rho\to 0}{G}/{\rho^k} = 0
$$
uniformly on the compact sets of $U$.
\begin{pro}\label{ag24}
With $U, \rho, S$ as above, let $f:S\to\HH$ be a smooth admissible function.
Then, there exists a smooth extension $G$ of $f$ to $U$ such
that $\overline{\mathfrak D}G$ is flat on $S$.
\end{pro}
\begin{proof}
By Proposition~\ref{ag19}, there exist smooth
functions $F:U\to\HH$, $\alpha_j:U\to\HH$, $\beta_j:U\to\HH^2$, ($j\ge 2$) such that
\begin{itemize}
\item $F|_S=f$;
\item (${\partial^k\alpha_j}/{\partial\rho^k})|_S=0$, $\forall k\ge1,j\ge2$;
\item $\overline{\mathfrak D}\big(F-\sum_{k=2}^m(\rho^k/k)\alpha_k\big)=\rho^m\beta_m$, $\forall m\ge 2$.
\end{itemize}
Moreover, by Lemma~\ref{ag18}, there exists a smooth
function $E:U\to\HH$ such that $E|_S=0$ and
$$
\frac{\partial^k E}{\partial\rho^k}\Big|_S={\frac{k!}{k+1}}{\alpha_{k+1}}|_S
$$
for all $k\ge 1$.

Let $T=\rho E$. Then, since $T|_S=0$ and
$$
\frac{\partial^kT}{\partial\rho^k}
=k\frac{\partial^{k-1}E}{\partial\rho^{k-1}}+\rho\frac{\partial^k E}{\partial\rho^k}
$$
for $k\ge 1$ in a neighborhood of $S$, we get
$$
{\alpha_k}|_S
=\frac{k}{(k-1)!}\frac{\partial^{k-1}E}{\partial\rho^{k-1}}\Big|_S
=\frac{1}{(k-1)!}\frac{\partial^k T}{\partial\rho^k}\Big|_S
$$
for all $k\ge 2$ and $\frac{\partial T}{\partial\rho}|_S=E|_S=0$.

Now, fix a point $p$ of $S$ and let $W_p$ be a neighborhood of $p$
where $\rho$ is one of the real coordinates, say the first,
and denote $\xi_1,\ldots,\xi_7$ the remaining.
Let $\pi:W_p\to W_p\cap S$ denote the
projection $(\rho,\xi_1,\ldots,\xi_7)\to (0,\xi_1,\ldots,\xi_7)$.
By what is preceding, we deduce that in $W_p$,
for all $m\ge 2$, the following holds true
$$
T-\sum_{k=2}^m\frac{\rho^k}{k}(\alpha_k\circ\pi)=
T-\sum_{k=0}^m\frac{\rho^k}{k!}\left(\frac{\partial^k T}{\partial\rho^k}\circ\pi\right)=\rho^{m+1}\zeta
$$
with $\zeta:W_p\to\HH$ smooth. Consequently,
\begin{equation}\label{ag25}
\overline{\mathfrak D}\left(T-\sum_{k=2}^m\frac{\rho^k}{k}(\alpha_k\circ\pi)\right)
=\rho^m v
\end{equation}
with $v:W_p\to\HH^2$ smooth. Moreover, since
$({\partial^k\alpha_j}/{\partial\rho^k})|_S=0, \forall k\ge1,j\ge2$, we
get
$$
\sum_{k=2}^m\frac{\rho^k}{k}\big(\alpha_k-\alpha_k\circ\pi\big)=\rho^{m+1}\theta,
$$
$\theta:W_p\to\HH$ smooth.
It follows that
\begin{equation}\label{ag26}
\overline{\mathfrak D}\left(\sum_{k=2}^m(\rho^k/k)\alpha_k-\sum_{k=2}^m(\rho^k/k)(\alpha_k\circ\pi)\right)
=\rho^m u,
\end{equation}
where $u:W_p\to\HH^2$ is smooth.

Finally, we define $G=F-T$.
Clearly $G|_S=f$ and by~\eqref{ag25},~\eqref{ag26}) we get
$$
\begin{aligned}
\overline{\mathfrak D}G
&=\phantom{-}\overline{\mathfrak D}\Big(F-\sum_{k=2}^m(\rho^k/k)\alpha_k\Big)
+\overline{\mathfrak D}\Big(\sum_{k=2}^m(\rho^k/k)\alpha_k-\sum_{k=2}^m(\rho^k/k)(\alpha_k\circ\pi)\Big)\\
&\phantom{=}\,-\overline{\mathfrak D}\Big(T-\sum_{k=2}^m(\rho^k/k)(\alpha_k\circ\pi)\Big)\\
&=\rho^m\big(\beta_m+u-v\big)=\rho^mw_p.
\end{aligned}
$$
Here $w_p:W_p\to\HH^2$ is smooth and uniquely determined by the
condition $\overline{\mathfrak D}G=\rho^mw_p$.
If $p\notin S$, we take $W_p$ such that $W_p\cap S=\varnothing$
and $w_p=\overline{\mathfrak D}G/\rho^m$.
Therefore, the family of the local maps $w_p$ defines a smooth 
map $w_m:U\to\HH^2$ such that $\overline{\mathfrak D}G=\rho^m w_m$ for every
integer $m\ge 2$, i.e.~$G$ is a smooth extension of $f$ to $U$ such
that $\overline{\mathfrak D}G$ is flat on $S$.

This proves Proposition~\ref{ag24}.
\end{proof}
We apply Proposition~\ref{ag24} in order to prove that the local
Riemann-Hilbert problem is always solvable.
This will follow from the following

\begin{teo}\label{RH}
Let $\Omega\subset\HH^2$ be a convex domain and $(S,U^+,U^-)$ a domain splitting of $\Omega$.
Let $f:S\to \HH$ a smooth admissible function.
Then, $f$ is a smooth jump.
\end{teo}
\begin{proof}
Observe that $S$ is orientable, so $S$ is defined by $\rho=0$,
where $\rho\in{\rm C}^\infty(\Omega)$.
Let $G:\Omega\to\HH$ a smooth extension of $f$,
with $\overline{\mathfrak D}G$ flat on $S$ (Proposition~\ref{ag24}).
Define $\eta:\Omega\to\HH^2$ by
$$
\eta=
\begin{cases}
-\overline{\mathfrak D}G\>\> & {\rm on}\>\> U^+\\
\>\>\>\>0\>\> & {\rm on}\>\> S\\
\>\>\>\overline{\mathfrak D}G\>\> & {\rm on}\>\> U^-.
\end{cases}
$$
$\eta$ is smooth in $\Omega$, since $\overline{\mathfrak D}G$ is flat on $S$.
Set $\eta=(\eta_1,\eta_2)$.
Then, the conditions
$$
\Delta_1\eta_2=\frac{\partial}{\partial\overline q_2}\frac{\partial\eta_1}{\partial q_1},\qquad
\Delta_2\eta_1=\frac{\partial}{\partial\overline q_1}\frac{\partial\eta_2}{\partial q_2}
$$
are satisfied on $U^+\cup U^-$ (see~\eqref{cr3}) whence on $\Omega$.
Since $\Omega$ is convex, there exists $\psi:\Omega\to\HH$ smooth such
that $\overline{\mathfrak D}\psi=\eta$~\cite{ABLSS}.
Defining $F^+=(\psi+G)/2$, $F^-=(\psi-G)/2$,
we have the following: $F^+$ and $F^-$ are smooth up to $S$,
$\overline{\mathfrak D}F^+=0$ ($\overline{\mathfrak D}F^-=0$)
in $U^+$ ($U^-$) and $F^+|_S-F^-|_S=f$.

This ends the proof of Theorem~\ref{RH}.
\end{proof}

\subsubsection{Two applications}

\begin{teo}\label{set25}
Let $\Omega$ be a bounded domain with connected smooth boundary ${\rm b}\,\Omega$.
Then, every smooth admissible function $f:{\rm b}\,\Omega \to \HH$ extends to
$\Omega$ by an $\HH$-holomorphic function, smooth up to ${\rm b}\,\Omega$.
\end{teo}

\begin{proof}
In our hypothesis,
$\HH^2 \setminus \overline\Omega$ is connected with boundary ${\rm b}\,\Omega$.
Since
$({\rm b}\,\Omega, \Omega, \HH^2 \setminus \overline\Omega)$ is a domain splitting of $\HH^2$,
by Theorem~\ref{RH},
$f = F^+|_S - F^-|_S$, where $F^+$, $F^-$ are $\HH$-holomorphic.
By Hartogs' Theorem $F^-$ extends to all of $\HH^2$ by an $\HH$-holomorphic
function $\widehat F^-$.
And this implies that $f$ is the boundary value
of $F^+ - \widehat F^-$.
\end{proof}

\begin{teo}\label{appl1}
Let $\Sigma$ be an open half-space and $S\subset \HH^2$ a
connected closed smooth hypersurface of $\Sigma$.
Assume that $\Sigma\setminus S$ splits into two connected
components $D$ and $W$ with $D$ bounded.
Let $f:S\to\HH$ be a smooth admissible function.
Then, $f$ extends to $D$ by an $\HH$-holomorphic
function $F$ which is smooth up to $S$.
\end{teo}
\begin{proof}
Without loss of generality, we can assume that $\Sigma$
be the half space $\{y_3> 0\}$.
Let $B$ be an open ball centered at origin such
that $S$ divides $B\cap\Sigma$ into two connected
components $U^+$ and $U^-=D$ and $D$ is relatively compact in $B$.
By Theorem~\ref{RH}, there are $\HH$-holomorphic
functions $F^+:U^+\to\HH$, $F^-:D\to\HH$, smooth up to $S$,
such that $f=F^+|_S-F^-|_S$.
It is enough to show that $F^+$ extends $\HH$-holomorphically to
$B\cap \{y_3>0\}$.
We may assume that $F^+$ is defined on an neighborhood
of ${\rm b}\,B \cap \Sigma$ in $\Sigma$.

Fix $\varepsilon > 0$ sufficiently small.
For every $c>0$, let $S_c$ be the sphere centered
at $(0,-c{\sf k})$ and passing through ${\rm b}\,B \cap \{y_3 = \varepsilon\}$.
Consider the set ${\mathcal C}$ of $c\in \RR$ such that
$F^+|_{S_c\cap \overline B}$ extends to a neighborhood of $\overline{S}_c\cap \overline B$ in $\overline B$.
We have ${\mathcal C}\ne \emptyset$.
Let $c_0 = \sup {\mathcal C}$,
and assume by contradiction that $c_0$ is finite.
Observe that $F^+$ is defined in a neighborhood of
${\rm b}B \cap \{y_3 = \varepsilon\}$ in $\overline B$.
Consider $B_{c_0}$, the open ball having $S_{c_0}$ as its boundary and
let $U = B \setminus \overline B_{c_0}$.
Then, the second fundamental form of $S_{c_0} \cap \Sigma$
(as part of the boundary of $U$) is negative definite.
Hence, as in the proof of Theorem~4 of~\cite{Pe3},
we get that, for every $q_0 \in S_{c_0} \cap \overline B$,
there exists a domain $\Delta_{q_0} \subset U$ such that
every $\HH$-holomorphic function in $\Delta_{q_0}$ extends to a bigger
domain $\widehat \Delta_{q_0}$ containing $q_0$.
It follows that $F^+$ extends $\HH$-holomorphically to a neighborhood
of $S_{c_0}\cap \overline B$ in $\overline B$: contradiction.
This means that $c_0=+\infty$, thus $F^+$ extends to $B\cap \{y_3 > \varepsilon\}$,
for every $\varepsilon$ near $0^+$.
By analytic continuation (see Theorem~\ref{sept4a} below), this completes the proof.
\end{proof}

\begin{rem}\label{appl2}
With the same notations of the above Theorem,
let $F$ be the $\HH$-holomorphic extension of $f$.
If $|f|$ is bounded on $S$, then for every $q\in D$
$$
|F(q)| \le \sup_S \,\lvert f\rvert.
$$
\end{rem}

We mention that, an extension theorem of different type,
has been recently found by Baracco, Fassina and Pinton~\cite{BFP}.

Let $S$ be a connected smooth hypersurface in $\HH^2$.
We say that the \emph{analytic continuation principle} holds
for smooth admissible functions
on $S$ when the following is true: if $f:S\to\HH$ is a smooth
admissible function which vanishes on a nonempty open set
of $S$, then $f\equiv 0$.
  
\begin{teo}\label{sept4a}
Let $S$ be a connected smooth hypersurface in $\HH^2$.
Then, the analytic continuation principle for smooth admissible functions holds
on $S$ in the following two cases:
\begin{itemize}
\item[i)] $S$ is the boundary of a domain $\Omega\Subset \HH^2$ satisfying the
  hypothesis of Theorem~\ref{set25};
\item[ii)] $S$ is nondegenerate.
\end{itemize}
\end{teo}

\begin{proof}
i)
Consider a smooth admissible function $f$ on $S$, $F$ its $\HH$-holomorphic
extension on $\Omega$, and let $Z = \{f = 0\}$.
Let $q_0 \in \mathring Z$
and $U$ be a neighborhood of $q_0$ relatively compact in $\mathring Z$.
Then there exists a domain $\Omega_1$ with smooth boundary,
satisfying the hypothesis of Theorem~\ref{set25},
such that $\Omega \subset \Omega_1$,
${\rm b}\, \Omega_1 \setminus {\rm b}\, \Omega \subset \HH^2 \setminus \overline \Omega$, and
${\rm b}\, \Omega \setminus U = {\rm b}\, \Omega_1 \cap {\rm b}\, \Omega$.
The function $f_1$ on ${\rm b}\, \Omega_1$ that
coincides with $f$ on ${\rm b}\, \Omega_1 \cap {\rm b}\, \Omega$ and is zero elsewhere,
is smooth admissible and, by Theorem~\ref{set25},
extends to an $\HH$-holomorphic function $F_1$ on $\Omega_1$,
smooth up to the boundary.
By the Bochner-Martinelli formula, it follows immediately
that $F_1$ is an extension of $F$.
By construction, $F_1$ vanishes on the boundary of $\Omega_1 \setminus \Omega$,
and then, $F_1$ vanishes on $\Omega_1 \setminus \Omega$.
By analytic continuation, $F_1\equiv 0$ and
therefore $F\equiv 0$ and $f\equiv 0$ too.

ii)
Let $f$ be a smooth admissible function on
$S = \{\rho = 0\}$ and let $Z = \{f = 0\}$.
Assume $f$ is not identically zero.
By Theorem~6~\cite{Pe3}, there exists a neighborhood $U$ of $S$ in,
say, $\{\rho \le 0\}$, such that the function $f$
extends by an $\HH$-holomorphic function $F$.
Take a point $p\in S$, there exists a domain $\Omega \subset \{\rho < 0\}$,
whose boundary contains $p$,
such that the set ${\rm b}\, \Omega \cap Z$ has interior points in $S$,
and the hypothesis of i) holds for $\Omega$.
Using i), $F|_{{\rm b}\,\Omega} = 0$, in particular, $f(p) = 0$.
This concludes the proof, $p$ being a generic point of $S$.
\end{proof}

\begin{rem}
The analytic continuation principle does not hold for an arbitrary
smooth hypersurface $S$.
For instance, all smooth functions $f=f(y_0,y_1,y_2)$
are admissible on $S=\{y_3=0\}.$
\end{rem}

\section{The CRF condition in weak form}\label{F1}

In order to treat the Riemann-Hilbert problem (in particular the boundary
problem) for $\HH$-holomorphic  functions with continuous boundary data we
need to give the {\rm CRF} conditions in a weak form.
We need some preliminaries.

Let $(S,U^+,U^-)$ be a domain splitting of a domain $\Omega$ in $\RR^n$
and $\varrho\in{\rm C}^\infty(\Omega)$ such that
$$
S = \{\rho = 0\},\qquad U^+ = \{\rho > 0\},\qquad U^- = \{\rho < 0\},\qquad \nabla \rho|_S \neq 0
$$
and consider on $S$ the orientation determined on the boundary
of $U^+$ by the inward normal vector.

By the existence of tubular neighborhoods we may assume that for $-\varepsilon0<\varepsilon<\varepsilon0$, $\varepsilon0>0$,
the hypersurface $S_\varepsilon=\{\varrho=\varepsilon\}$ is diffeomorphic to $S$ by a diffeomorphism $\pi_\varepsilon: S_\varepsilon\to S$.

Let $T$ be a distribution on $S$. We say that $T$ is the \emph{trace} or the \emph{boundary value} (in the sense of distributions) of a function  $u\in {\rm L}^1_{\rm loc}(U^+)$  following $\{S_\varepsilon\}_{0<\varepsilon<\varepsilon0}$ if 
$$
\lim\limits_{\varepsilon\to 0^+}\int_{S_\varepsilon}\pi_\varepsilon^\ast(\phi)u=(T,\phi)
$$
for every real-valued test form $\phi$ of class ${\rm C}_0^\infty$
on $S$ of degree $n-1$.  In such a situation we set $\gamma^+(u)=T$. 

In the same manner we give the notion of trace $\gamma^-(u)$ if $u\in {\rm L}^1_{\rm loc}(U^-)$.
 
The following result was proved in \cite[Corollary~I.~2.~6]{LT}.
Let $P(D)$ be a linear elliptic operator on $U^+$ with smooth coefficients
and $u\in{\rm C}^\infty(U^+)$ a solution of the equation $P(D)u=0$.
Then $u$ has a boundary value $\gamma^+(u)$ if and only if $u$ extends as distribution through $S$.

Now we are in position to state the CRF condition in a weak form.                                                                                                                                                                    
\begin{pro}\label{lu24}
Let $f:S\to\HH$ be continuous function and $T_1=T_{1,f}$, $T_2=T_{2,f}$ the
distributions on $\Omega$ \(supported by $S$\) defined by  
\begin{equation}\label{lu73}
\begin{cases}
\phi\mapsto\int_S{{\rm D}q_1}\wedge f\phi\,{\rm d}y\\
\phi\mapsto\int_S{{\rm D}q_2}\wedge f\phi\,{\rm d}x,
\end{cases}
\end{equation}
(where $\phi\in{\rm C}_0^\infty(\Omega)$ is a real valued test function).
If $f$ is locally a jump of $\HH$-holomorphic functions continuous up $S$,
then the system
\begin{equation}\label{lu25}
\left\{
\begin{aligned}
\frac{\partial v}{\partial\overline q_1}&=T_1\\
\frac{\partial v}{\partial\overline q_2}&=T_2
\end{aligned}
\right.
\end{equation}
is locally solvable along $S$.
\end{pro}
\begin{proof}
Let $q_0\in S$ and $F^\pm$ $\HH$-holomorphic
functions in $U^\pm$, smooth up to $S$ such
that $f|_{U\cap S}={F^+}|_S-{F^-}|_S$.
Let
\[
F=\begin{cases}
- F^+ \>\> {\rm in}\>U^+\\
- F^- \>\> {\rm in}\>U^-
\end{cases}
\]
and denote by the same letter $F$ the distribution
$$
\phi\mapsto\int_\Omega \phi\,F\,{\rm d}x\wedge{\rm d}y:
$$
here $\phi$ is a real valued test function.
Then
\begin{equation}\label{lu27}
\begin{aligned}
\frac{\partial F}{\partial\overline q_1}(\phi)
&= -\int_\Omega \phi_{\overline q_1}\,F\,{\rm d}x\wedge{\rm d}y\\
&= \int_{U^+} \phi_{\overline q_1}\,F^+{\rm d}x\wedge{\rm d}y
 + \int_{U^-} \phi_{\overline q_1}\,F^-{\rm d}x\wedge{\rm d}y.
\end{aligned}
\end{equation}
Denote ${\rm d}_x$ (${\rm d}_y$) the differential with
respect to the $x$ ($y$)-variables.
Then, since ${\rm D}q_1$ is closed, we have
$$
\begin{aligned}
\int_{U^+} \phi_{\overline q_1}\,F^+{\rm d}x\wedge{\rm d}y
&=\int_{U^+}{\rm d}_x({\rm D}q_1\phi) F^+\wedge {\rm d}y\\
&=\int_{U^+}{\rm d}_x({\rm D}q_1\cdot F^+\cdot\phi)\wedge {\rm d}y
  -\int_{U^+}{\rm d}_x({\rm D}q_1\cdot F^+)\phi\wedge {\rm d}y\\
&=\int_{U^+}{\rm d}({\rm D}q_1\cdot F^+\phi \wedge {\rm d}y)
  -\int_{U^+}{\rm d}_x({\rm D}q_1\cdot F^+)\phi \wedge {\rm d}y\\
&=\int_{S}{\rm D}q_1\cdot F^+\phi \wedge {\rm d}y
  -\int_{U^+}\frac{\partial{F^+}}{\partial\bar q_1}\phi\,{\rm d}x\wedge{\rm d}y\\
&=\int_{S}{\rm D}q_1\cdot F^+\phi \wedge {\rm d}y
\end{aligned}
$$
by~\eqref{lu1} and the $\HH$-holomorphy of $F^+$.

In the same manner,
$$
\int_{U^-} \phi_{\overline q_1}F^-{\rm d}x\wedge{\rm d}y
=-\int_{S}{\rm D}q_1\cdot F^-\phi \wedge {\rm d}y
$$
(${\rm b}\,U^-\cap S=-{\rm b}\,U^+\cap S$),
whence
\begin{equation}\label{lu28}
\begin{aligned}
\frac{\partial F}{\partial\overline q_1}(\phi)
&=\int_{S}{\rm D}q_1\cdot F^+\phi\wedge {\rm d}y
  -\int_{S}{\rm D}q_1\cdot F^-\phi\wedge {\rm d}y\\
&=\int_{S}{\rm D}q_1\cdot f\phi \wedge {\rm d}y.
\end{aligned}
\end{equation}
Analogously, we get
\begin{equation}\label{lu29}
\frac{\partial F}{\partial\overline q_2}(\phi)=\int_{S}{\rm D}q_2\cdot f\phi \wedge {\rm d}x.
\end{equation}
Equations~\eqref{lu28} and~\eqref{lu29} show that the distribution $F$ is a solution of~\eqref{lu25}.

If $f$ is only continuous, we approximate $S$ by
hypersurfaces $S_\varepsilon$, with $0 < \lvert \varepsilon \rvert < \varepsilon0$, and
we apply the previous argument to $F^+$ on $\{ \rho \ge \varepsilon \}$ when $\varepsilon>0$,
and $F^-$ on $\{ \rho \le \varepsilon \}$ when $\varepsilon<0$.
Hence, by taking the limits, identities~\eqref{lu28} and~\eqref{lu29}
holds in the continuous case too.
\end{proof}

From the above Proposition and~\eqref{cr2}, it follows
\begin{cor}\label{lu30}
If $f$ is a jump of $\HH$-holomorphic functions continuous up to $S$,
then in $\Omega$ we have
\begin{equation}
  \label{lu32}
  \begin{aligned}
    \frac{\partial}{\partial \overline q_1}\frac{\partial T_2}{\partial q_2}-\Delta_2 T_1&=0\\
    \frac{\partial}{\partial \overline q_2}\frac{\partial T_1}{\partial q_1}-\Delta_1 T_2&=0          
  \end{aligned}
\end{equation}
in the distribution sense, i.e.,
\begin{equation}
  \label{33}
  \begin{aligned}
    \int_S\left[\Big(\frac{\partial}{\partial \overline q_1}\frac{\partial}{\partial q_2}\phi\Big){\rm d}x\wedge{\rm D}q_2-(\Delta_2\phi){\rm d}y\wedge{\rm D}q_1 \right]f=0\\
    \int_S\left[\Big(\frac{\partial}{\partial \overline q_2}\frac{\partial }{\partial q_1}\phi\Big){\rm d}y\wedge{\rm D}q_1-(\Delta_1\phi){\rm d}x\wedge{\rm D}q_2\right]f=0        
  \end{aligned}
\end{equation}
for every $\phi\in{\rm C}_0^\infty(\Omega)$.
\end{cor}

We say that a continuous function $f:S\to\HH$ is
a \emph{weakly admissible function} if it satisfies~\eqref{33}.
 
We have the
\begin{teo}\label{lu33}
A continuous function $f:S\to\HH$ is locally a jump
of $\HH$-holomorphic functions, continuous up
to $S$ if and only if it is a weakly admissible function.
In particular, assume that $S$ is the connected boundary of a bounded domain $\Omega$.
Then, every continuous admissible function $f:{\rm b}\,\Omega\to\HH$ extends
to $\Omega$ by an $\HH$-holomorphic function which is continuous up to ${\rm b}\,\Omega$.
\end{teo}

\begin{proof}
We only have to prove that if $f$ is weakly admissible then
it is locally a jump of $\HH$-holomorphic functions continuous up to $S$.

Let $q_0\in S$, and $(S\cap\Omega,U^+,U^-)$ be a splitting domain of a convex
domain $\Omega$ containing $q_0$.
Since $f$ is weakly admissible and $\Omega$ is convex,
by~\eqref{lu32} and Remark~\ref{14ott} there exists a
distribution $F$ in $\Omega$ such that
\begin{equation}\label{lu95}
\left\{
\begin{aligned}
\frac{\partial F}{\partial\overline q_1}&=T_1\\
\frac{\partial F}{\partial\overline q_2}&=T_2.
\end{aligned}
\right.
\end{equation}
Since
$\frac{\partial F}{\partial\overline q_1}|_{\Omega\setminus S}=\frac{\partial F}{\partial\overline q_2}|_{\Omega\setminus S}=0$,
$F$ is $\HH$-holomorphic on $\Omega\setminus S$.

Let $F^\pm=F|_{U^\pm}$.
Since $F^\pm$ are pluriharmonic, the results of~\cite{LT} apply.
In particular,
$F^\pm$ have traces $\gamma(F^\pm)$ on $S$ in the sense of
distributions and $\gamma(F^+)-\gamma(F^-)=f$~\cite[Corollaire~I.2.6 and Théorème~II.1.3]{LT}.

Let $V^\pm$ be domains with smooth boundary such
that $V^\pm\Subset \Omega$, $V^\pm\subset U^\pm$
and $({\rm b}V^+\cap S)=({\rm b}V^-\cap S)$ is a relative
neighborhood $S_0$ of $q_0$ in $S$.
Again, by~\cite[Corollaire~I.2.6]{LT}, $F^\pm$ have traces $\theta^\pm$
on ${\rm b}V^\pm$ in the sense of distributions, and, by the Bochner-Martinelli
formula for $\HH$-holomorphic functions, we have
\[
F^\pm(q)=\big\langle\theta^\pm,{\bf K}^{BM}(\cdot,q)\big\rangle
\]
for $q\in V^\pm$.

Let $\psi\in{\rm C}_0^\infty(S_0)$ such that $\psi=1$ on a neighborhood of $q_0$.
Then,
$$
\begin{aligned}
\big\langle\psi\theta^+,{\bf K}^{BM}(\cdot,q)\big\rangle+\big\langle(1-\psi)\theta^+,{\bf K}^{BM}(\cdot,q)\big\rangle
&=
\begin{cases}
F^+(q)\>\>\>q\in V^+\\
\>0 \>\>\>\>\>\>\>\>\>\>\>q\notin\overline V^+
\end{cases}
\\
\big\langle\psi\theta^-,{\bf K}^{BM}(\cdot,q)\big\rangle+\big\langle(1-\psi)\theta^-,{\bf K}^{BM}(\cdot,q)\big\rangle
&=
\begin{cases}
F^-(q)\>\>\>q\in V^-\\
\>0 \>\>\>\>\>\>\>\>\>\>\>q\notin\overline V^-.
\end{cases}
\end{aligned}
$$
The functions $\big\langle(1-\psi)\theta^\pm,{\bf K}^{BM}(\cdot,q)\big\rangle$,
$\big\langle \psi\theta^\pm,{\bf K}^{BM}(\cdot,q)\big\rangle$ are smooth
near $q_0$ and, since $\theta^+-\theta^-=f$, we have
$$
\begin{aligned}
  F^+(q)
  &= \big\langle\psi f,{\bf K}^{BM}(\cdot,q)\big\rangle +
     \big\langle\psi \theta^-,{\bf K}^{BM}(\cdot,q)\big\rangle+\big\langle(1-\psi)\theta^+,{\bf K}^{BM}(\cdot,q)\big\rangle\\
  F^-(q)
  &= - \big\langle\psi f,{\bf K}^{BM}(\cdot,q)\big\rangle +
    \big\langle(\psi\theta^+,{\bf K}^{BM}(\cdot,q)\big\rangle+\big\langle(1-\psi)\theta^-,{\bf K}^{BM}(\cdot,q)\big\rangle
\end{aligned}
$$
and consequently
$$
\begin{aligned}
  F^+(q)&=\phantom{-}\int_{{\rm b}V^+}\psi f{\bf K}^{BM}(\cdot,q)+u(q)\\
  F^-(q)&= - \int_{{\rm b}V^-} \psi f{\bf K}^{BM}(\cdot,q)+v(q)
\end{aligned}
$$
with $u=u(q)$, $v=v(q)$ smooth near $q_0$.
Now, as a consequence of the classical
potential theory~\cite{Mi}, $F^\pm$ are continuous up to
the boundary and this concludes the proof of the general case.
In the particular case when $S$ is the boundary of $\Omega$, the proof runs as
in the smooth case of Theorem~\ref{set25}.
\end{proof}

\begin{rem}\label{sept3}
Theorem~\ref{appl1} also generalizes.
\end{rem}

\begin{rem}\label{sept23}
A smooth function $f:S\to\HH$ is weakly admissible if and only if is admissible.
\end{rem}

\begin{proof}
First assume that $f$ is $C^\infty$ and weakly admissible.
By Remark~\ref{rmk:restrict-crf} the
functions $F^\pm$ of the previous Proposition are
smooth up to the boundary $S$, hence $f$ is admissible.

Next, assume $f$ is $C^\infty$ and admissible, then $f = F^+ - F^-$.
Since $F^\pm$ are smooth up to $S$, then $f$ is weakly admissible.
\end{proof}

Let $S$ be a connected smooth hypersurface in $\HH^2$.
We say that the \emph{analytic continuation principle} holds
for weakly admissible functions on $S$ when the following is
true: if $f:S\to\HH$ is a continuous
weakly admissible function which vanishes on an nonempty open set of $S$,
then $f\equiv 0$.

\begin{teo}\label{sept4b}
Let $S$ be a connected smooth hypersurface in $\HH^2$.
Then, analytic continuation principle holds for weakly admissible
functions on $S$ in the following two cases:
\begin{itemize}
\item[i)] $S$ is the boundary of a domain $\Omega\Subset \HH^2$ satisfying the
  hypothesis of Theorem~\ref{set25};
\item[ii)] $S$ is nondegenerate.
\end{itemize}
\end{teo}
\begin{proof}
The proof is analogous to the one of Theorem~\ref{sept4a}
using Theorem~\ref{lu33} instead of Theorem~\ref{set25}.
\end{proof}

\begin{pro}\label{sept2}
Let $S=\{\rho=0\}$ be a smooth hypersurface in $\HH^2$,
$\nabla \rho \neq 0$ on $S$, and $\Omega^-=\{\rho<0\}$.
Assume that $\Omega^-$ is strongly Levi $\HH$-convex along $S$.
Then, every weakly admissible function $f:S\to\HH$ extends
to a neighborhood $U$ of $S$ in $S \cup \Omega^-$ by an
$\HH$-holomorphic function in $U$, continuous up to $S$.
\end{pro}

\begin{proof}
By Theorem~\ref{lu33}, using Kontinuitätssatz
as in the proof of Theorem~4 of~\cite{Pe3},
for every point of $p\in S$ there exists a ball $B(p)$ such
that $f|_{B(p)\cap \{\rho = 0 \}}$ extends $\HH$-holomorphically
on $B(p) \cap \Omega^-$ by a function $F_p$.
This implies that there exists an open covering $B(p_j)$ of $S$
and $\HH$-holomorphic functions $F_j : B(p_j) \cap \Omega^- \to \HH$,
continuous up to $S$, such that $F_j$ and $f$ agree on $B(p_j) \cap S$.
By construction, $F_j = F_k$ on the
intersection $B(p_j) \cap B(p_k) \cap S$,
hence, by the analytic continuation principle, $F_j = F_k$
on $B(p_j) \cap B(p_k) \cap \Omega^-$.  Thus the functions $\{F_k\}$ defines
the required extension of $f$.
\end{proof}

\section{Appendix: Some generalizations to octonions}

We sketch some generalizations of our results to octonian regular functions.
We denote by $\ee_0=1$ the real unit and by $\ee_1,\dots,\ee_7$ the imaginary units of the division
algebra of the octonions $\OO$.  Thus, every element $p$ of $\OO$ can
be written in the form
\[
  p = \sum_{\alpha = 0}^7 x_\alpha \ee_\alpha
  \qquad\text{with $x_\alpha\in \RR$.}
\]
As usual, we set $\RE(p)=x_0$, $\IM(p) = \sum_{\alpha = 1}^7 x_\alpha \ee_\alpha$ and $\bar p = \RE(p) - \IM(p)$.
We recall that the product of octonions is noncommutative and nonassociative.

Let $U$ be an open set in $\OO$ and $u:U\to \OO$ a
smooth function. The Cauchy-Riemann-Fueter
operator $\partial_{\bar p}$ acts on $u$ in the following way:
\[
  \partial_{\bar p}u =
  \sum_{\alpha = 0}^7 \ee_\alpha\frac{\partial u}{\partial x_\alpha} =
  \left(\sum_{\alpha = 0}^7 \ee_\alpha\partial_{x_\alpha}\right) u.
\]
We say that $u$ is \emph{(left) $\OO$-holomorphic} in $U$ if $\partial_{\bar p}u = 0$ on $U$.
We also consider the conjugate operator
\[
  \partial_p u =
  \overline{\partial_{\bar p}}u =
  \sum_{\alpha = 0}^7 \bar\ee_\alpha\frac{\partial u}{\partial x_\alpha} =
  \left(\sum_{\alpha = 0}^7 \bar \ee_\alpha\partial_{x_\alpha}\right) u.
\]
In the case of several octonian variables $p_1,\dots,p_n$, we set
\begin{equation}
  \label{eq:octonion}
  p_h = \sum_{\alpha= 0}^7 x_{h,\alpha} \ee_\alpha,
  \qquad\text{with $x_{h,\alpha} \in \RR$,}
\end{equation}
and given an open subset $U$ of $\OO^n$, we consider the set $\calE^r(U)$ of the
smooth maps $U\to\OO^r$.

Let us consider a function $u \in \calE^1(U)$, $u = u(p_1,\dots,p_n)$.
We define the operator
\begin{equation}
  \label{eq:3}
  \overline{\mathfrak{D}} u = 
  \left(\partial_{\bar p_1} u,\dots, \partial_{\bar p_n} u\right).
\end{equation}
We have $\overline{\mathfrak{D}} u \in \calE^n(U)$.  The kernel of the operator
$\overline{\mathfrak{D}}$ consists of the \emph{(left) $\OO$-holomorphic functions} in the sense of
Fueter.

For some of the basic results in octonian analysis, we refer to
\cite{DS,li-peng-2002,Wang2014}. 

Lef $f=(f_1,\dots,f_n)\in\calE^n(U)$.
The \emph{non-homogeneous Cauchy-Riemann-Fueter problem} asks for
the existence of a solution of
\begin{equation}
  \label{eq:crf}
  \overline{\mathfrak{D}} u = f
\end{equation}
that is
\begin{equation}
  \label{eq:crf'}
  \partial_{\bar p_h} u = f_h
\end{equation}
for $h=1,\dots n$.

In this Appendix, we aim to study conditions on
$U\subseteq\OO^2$ and $f=(f_1,f_2)$
which guarantee the existence of a solution $u\in\calE^1(U)$
of~\eqref{eq:crf}.
In other words, to characterize the image of the
operator $\overline{\mathfrak{D}}$ for $n = 2$.

We start by looking at the necessary conditions on the datum $f$ for
arbitrary $n$.  Let us recall that
$\partial_{p_m}\partial_{\bar p_m} =
 \partial_{\bar p_m}\partial_{p_m} = \Delta_{p_m}$
is the laplacian with respect to the real coordinates of the
octonian variable $p_m$.
If the system of Cauchy-Riemann-Fueter~\eqref{eq:crf} has a solution,
the datum $f$ must satisfy the equations
\begin{equation}
  \label{eq:compat3}
  \Delta_{p_m} f_l
  = \partial_{\bar p_l}(\partial_{p_m} f_m)
\end{equation}
for $l, m = 1,\dots,n$.
Indeed, if $u$ is a solution of~\eqref{eq:crf},
i.e,~$\partial_{\bar p_i} u = f_i$, $i=1,\dots n$.
Then
\begin{equation*}
  \Delta_{p_m} f_l =
  \Delta_{p_m} (\partial_{\bar p_l} u) =
  \partial_{\bar p_l} (\Delta_{p_m} u) =
  \partial_{\bar p_l} (\partial_{p_m}(\partial_{\bar p_m} u)) =
  \partial_{\bar p_l}(\partial_{p_m} f_m).
\end{equation*}
Wang and Ren proved in~\cite{Wang2014} that such conditions are
actually sufficient when the data $f_1,\dots, f_n$ have a compact support.
For the sufficience in the general case we follow the method of
Ehrenpreis~\cite{Eh}.

Once written the system~\eqref{eq:crf} in the form
$\overline{D}u = f$, where $\overline{D}$ is the real matrix of
differential operator $\overline{\mathfrak{D}}$, the problem reduces to find the
generators of the module of relations of the rows of
$\overline{D}$.

In real coordinates, given the function
\[
  u(p) = \sum_{\alpha = 0}^7 u_\alpha(p) \ee_\alpha = (u_0(p),\dots,u_7(p)),
\]
the Cauchy-Riemann-Fueter operator (in one variable) takes the form
\begin{equation}
  \label{eq:2}
  \overline{\mathfrak{D}} u =
  \overline D
  \begin{bmatrix}
    u_0 \\ \vdots \\ u_7
  \end{bmatrix}
  =
  \begin{bmatrix}
    \partial_{x_0} &          - \partial_{x_1} &          - \partial_{x_2} &          - \partial_{x_3} &          - \partial_{x_4} &          - \partial_{x_5} &          - \partial_{x_6} &          - \partial_{x_7} \\
    \partial_{x_1} & \phantom{-}\partial_{x_0} &          - \partial_{x_3} & \phantom{-}\partial_{x_2} &          - \partial_{x_5} & \phantom{-}\partial_{x_4} & \phantom{-}\partial_{x_7} &          - \partial_{x_6} \\
    \partial_{x_2} & \phantom{-}\partial_{x_3} & \phantom{-}\partial_{x_0} &          - \partial_{x_1} &          - \partial_{x_6} &          - \partial_{x_7} & \phantom{-}\partial_{x_4} & \phantom{-}\partial_{x_5} \\
    \partial_{x_3} &          - \partial_{x_2} & \phantom{-}\partial_{x_1} & \phantom{-}\partial_{x_0} &          - \partial_{x_7} & \phantom{-}\partial_{x_6} &          - \partial_{x_5} & \phantom{-}\partial_{x_4} \\
    \partial_{x_4} & \phantom{-}\partial_{x_5} & \phantom{-}\partial_{x_6} & \phantom{-}\partial_{x_7} & \phantom{-}\partial_{x_0} &          - \partial_{x_1} &          - \partial_{x_2} &          - \partial_{x_3} \\
    \partial_{x_5} &          - \partial_{x_4} & \phantom{-}\partial_{x_7} &          - \partial_{x_6} & \phantom{-}\partial_{x_1} & \phantom{-}\partial_{x_0} & \phantom{-}\partial_{x_3} &          - \partial_{x_2} \\
    \partial_{x_6} &          - \partial_{x_7} &          - \partial_{x_4} & \phantom{-}\partial_{x_5} & \phantom{-}\partial_{x_2} &          - \partial_{x_3} & \phantom{-}\partial_{x_0} & \phantom{-}\partial_{x_1} \\
    \partial_{x_7} & \phantom{-}\partial_{x_6} &          - \partial_{x_5} &          - \partial_{x_4} & \phantom{-}\partial_{x_3} & \phantom{-}\partial_{x_2} &          - \partial_{x_1} & \phantom{-}\partial_{x_0}
  \end{bmatrix}
  \begin{bmatrix}
    u_0 \\ u_1 \\ u_2 \\ u_3 \\ u_4 \\ u_5 \\ u_6 \\ u_7
  \end{bmatrix}
\end{equation}
(see also \cite{li-peng-2002}).
Analogously, in the multivariate case, the Cauchy-Riemann-Fueter operator
can be written in real components as
\[
  \overline{\mathfrak{D}} u =
  \overline D
  \begin{bmatrix}
    u_0 \\ \vdots \\ u_7
  \end{bmatrix},
\]
where
\begin{equation}
  \label{eq:matriceD}
  \overline D =
  \begin{bmatrix}
    \overline D_{p_1} \\
    \vdots       \\
    \overline D_{p_n}
  \end{bmatrix}
\end{equation}
is a $8n\times 8$ matrix with entries in the polynomial ring with $8n$
indeterminates
\[
  R_n =
  \RR[\partial_{x_{1,0}}, \dots, \partial_{x_{1,7}}, \dots,
      \partial_{x_{n,0}}, \dots, \partial_{x_{n,7}}],
\]
and $\overline D_{p_i}$ denotes the matrix $\overline D$ relative to the variable $p_i$.

We denote by $\Syz$ the module of syzygies of the rows of the
matrix $\overline D$, which is a graded module with grading inherited
by the polynomial ring $R_n$.
By taking the real components, we get eight such real syzygies from each
one of the octonian conditions~\eqref{eq:compat3}.

\begin{pro}
  \label{prop:condizioni-indipendenti}
  The $n(n-1)$ conditions $\Delta_{p_m} f_l = \partial_{\bar p_l}(\partial_{p_m} f_m)$,
  of~\eqref{eq:compat3} for $l,m=1,\dots n$, $l\neq m$, give $8n(n-1)$
  real quadratic relations.  These relations corresponds to linearly
  independent elements over $\RR$ in $\Syz$.
\end{pro}

\begin{proof}
  We want to prove that the operators
  $z_{l,m}(f) = \Delta_{p_m} f_l - \partial_{\bar p_l}(\partial_{p_m} f_m)$
  for $l,m = 1,\dots n$, $l\neq m$ are linear independent on $\RR$.
  Given $a,b=1,\dots,n$, $a\neq b$, we will prove that
  $z_{a,b}$ is not a linear combination of the other $z_{l,m}$.
  Indeed, consider the test data $g = (g_1,\dots, g_n)$ where
  \begin{equation}
    g_k(p_1,\dots,p_n) =
    \begin{cases}
      x_{b,0}^2 & \text{for $k = a$}, \\
      0 & \text{otherwise}
    \end{cases}
  \end{equation}
  with notations as in~\eqref{eq:octonion}.
  Then, $z_{l,m}(g)$ is nonzero if and only if $(l,m) = (a,b)$.
\end{proof}

Now we focus on the case of $n=2$ octonian variables $p_1,p_2$.
Conditions~\eqref{eq:compat3} become
\begin{equation*}
  \begin{aligned}
    \Delta_{p_2} f_1 &= \partial_{\bar p_1}(\partial_{p_2} f_2), \\
    \Delta_{p_1} f_2 &= \partial_{\bar p_2}(\partial_{p_1} f_1).
    \end{aligned}
\end{equation*}
Using a computer program that calculates the generators and the Betti
numbers of a graded module, one checks directly
that the module $\Syz$ is generated in degree $2$,
and $\Syz_2$, its component of degree $2$,
has real dimension $16$.\footnote{We performed the mentioned computation,
as well as the one of Remark~4.4, using the commands \texttt{syz} and
\texttt{betti} of the computer algebra system Macaulay~2~\cite{M2}.}
From this we get
\begin{pro}
  \label{prop:basis2}
  For $n=2$ octonian variables, the conditions~\eqref{eq:compat3}
  correspond to $16$ real relations that form a basis of the module
  of syzygies $\Syz$ as a real vector space.
\end{pro}

\begin{proof}
  It follows immediately from the above computer verification and
  Proposition~\ref{prop:condizioni-indipendenti}.
\end{proof}

Proposition~\ref{prop:basis2} and Ehrenpreis'
  Theorem~\cite[Theorem~6.2, p.~176]{Eh} now imply
\begin{teo}
  \label{th:sufficient}
  Let $U\subset \OO^2$ be a convex domain and $f\in\calE^2(U)$.
  Then, the Cauchy-Riemann-Fueter Problem
  $\overline{\mathfrak{D}}u = f$
  has a solution $u\in\calE^1(U)$ if and only if $f$ satisfies
  conditions~\eqref{eq:compat3}.
\end{teo}

\begin{rem}
  We stress that conditions~\eqref{eq:compat3} do not generate
  the module of syzygies $\Syz$ for $n>2$.
  For $n=3$, this can be directly checked
  (employing a computer algebra system),
  hence conditions~\eqref{eq:compat3} are not sufficient to guarantee
  the existence of a solution to~\eqref{eq:crf}.
\end{rem}

As in the quaternionic case, we can introduce the notion
of admissible octonian function.
Thus, in view of Theorem~\ref{th:sufficient}, we
can run through the proofs of Theorems~\ref{RH} and \ref{set25} and we get:
\begin{teo}
  Let $\Omega \subset \OO^2$ be a domain.
  \begin{enumerate}
  \item If $\Omega$ is convex and $(S,U^+,U^-)$ is a domain splitting of $\Omega$,
    then every smooth admissible function $f:S\to \OO$ is a smooth jump.
  \item If $\Omega$ is bounded with connected smooth boundary ${\rm b}\,\Omega$,
    then every smooth admissible function $f:{\rm b}\,\Omega \to \OO$ extends
    to $\Omega$ by an $\OO$-holomorphic function, smooth up to ${\rm b}\,\Omega$.
  \end{enumerate}
\end{teo}

\providecommand{\bysame}{\leavevmode\hbox to3em{\hrulefill}\thinspace}
\renewcommand{\MR}[1]{}
\providecommand{\MRhref}[2]{%
  \href{http://www.ams.org/mathscinet-getitem?mr=#1}{#2}
}
\providecommand{\href}[2]{#2}

\end{document}